\documentclass{amsart}
\usepackage{amscd,amssymb,amsopn,amsmath,amsthm,graphics,amsfonts,enumerate,verbatim,calc}
\usepackage[dvips]{graphicx}
\usepackage[colorlinks=true,linkcolor=red,citecolor=blue]{hyperref}
\usepackage[all]{xy}

\usepackage{mathrsfs}

\calclayout

\newcommand{\rt}{\rightarrow}
\newcommand{\lrt}{\longrightarrow}

\newcommand{\st}{\stackrel }

\newcommand{\la}{\lambda}
\newcommand{\La}{\Lambda}

\newcommand{\Z}{\mathbb{Z} }

\newcommand{\CC}{\mathcal{C} }

\newcommand{\CS}{\mathcal{S} }

\newcommand{\CH}{\mathcal{H}}

\newcommand{\Y}{\mathbf{Y}}
\newcommand{\W}{\mathbf{W}}

\newcommand{\mmod}{{\rm{{mod\mbox{-}}}}}

\newcommand{\im}{{\rm{Im}}}

\newcommand{\Ker}{{\rm{Ker}}}

\newcommand{\Hom}{{\rm{Hom}}}

\theoremstyle{plain}
\newtheorem{theorem}{Theorem}[section]
\newtheorem{corollary}[theorem]{Corollary}
\newtheorem{lemma}[theorem]{Lemma}

\newtheorem{proposition}[theorem]{Proposition}

\theoremstyle{definition}

\newtheorem{example}[theorem]{Example}

\newtheorem{construction}[theorem]{Construction}

\newtheorem{remark}[theorem]{Remark}

\theoremstyle{plain}

\theoremstyle{definition}

\numberwithin{equation}{section}

\begin{document}

\title[Determination of some almost split sequences in morphism categories]{Determination of some almost split sequences in morphism categories}

\author[Rasool Hafezi and Hossein Eshraghi  ]{Rasool Hafezi and Hossein Eshraghi }

\address{School of Mathematics and Statistics,
Nanjing University of Information Science \& Technology, Nanjing, Jiangsu 210044, P.\,R. China}
\email{hafezi@nuist.edu.cn \\ hafezi@ipm.ir}

\address{Department of Pure Mathematics, Faculty of Mathematical Sciences, University of Kashan, PO Box 87317-51167, Kashan, Iran}
\email{eshraghi@kashanu.ac.ir }



\begin{abstract}
Almost split sequences lie in the heart of Auslander-Reiten theory. This paper deals with the structure of almost split sequences with certain ending terms in the morphism category of an Artin algebra $\Lambda$. Firstly we try to interpret the Auslander-Reiten translates of particular objects in the morphism category in terms of the Auslander-Reiten translations within the category of $\Lambda$-modules, and then use them to calculate almost split sequences. In classical representation theory of algebras, it is quite important to recognize the midterms of almost split sequences. As such, another part of the paper is devoted to discuss the midterm of certain almost split sequences in the morphism category of $\Lambda$.  As an application, we restrict in the last part of the paper to self-injective algebras and present a structural theorem that illuminates a link between representation-finite morphism categories and Dynkin diagrams.
\end{abstract}

\maketitle
\section{Introduction}
For a ring $\Lambda$, its morphism category (resp. monomorphism category) $H(\Lambda)$ (resp. $S(\Lambda)$) has as its objects all $\Lambda$-maps (resp. all $\Lambda$-monomorphisms)
$(A\st{f}\rt B)$ in the category $\mmod\Lambda$ of all finitely generated left $\Lambda$-modules, and a morphism $h$ from an object $(A\st{f}\rt B)$ to $(C\st{g}\rt D)$ is given by a pair
$h=\left(\begin{smallmatrix}
h_1\\
h_2
\end{smallmatrix}
\right)$ of $\Lambda$-maps where $h_1: A\rt C$ and $h_2: B\rt D$ satisfy $g h_1=h_2 f$.
It is known that $H(\Lambda)$ is naturally equivalent to the category of finitely generated right modules over the triangular matrix ring $T_2(\Lambda)=\left(\begin{smallmatrix}
 \Lambda & \Lambda\\
 0 &\Lambda
 \end{smallmatrix}\right)$.

\vspace{.1 cm}

During last decades, there has been a rising interest in the monomorphism category of a ring. This interest comes from several origins among which the work of Birkhoff in 1934 \cite{Bi} proposing to consider the embeddings of a subgroup of a given Abelian group is worth mentioning. Thereafter, links to other areas of algebra such as determination of invariant subspaces of linear operators \cite{RS2} and Gorenstein homological algebra \cite{Z, LZ, ZX} were also detected.
In contrast to the monomorphism category (also called the submodule category), the morphism category itself has not been the target of many considerations. Probably this is due to the fact that, compared to $S(\Lambda)$, it has a fairly complicated nature particularly from representation-theoretic points of view.

\vspace{.1 cm}

The morphism category of an Artin algebra $\Lambda$ deserves special attention also from the functor category perspective. Namely, let $\mmod(\mmod\Lambda)$ be the category of all finitely presented (or coherent) contravariant additive functors on $\mmod\Lambda$ with values in $\mathbf{Ab}$, the category of Abelian groups. A classical result due to Auslander \cite {Au} suggests that one may recover the category $\mmod\Lambda$ as a quotient of $\mmod(\mmod\Lambda)$ via the equivalence $$\mmod\Lambda\simeq\frac{\mmod(\mmod\Lambda)}{\{ F\mid F(\Lambda)=0\}};$$ this expression is now known as Auslander Formula. The privilege of such an insight is that  $\mmod(\mmod\Lambda)$ enjoys nice homological properties that may facilitate some arguments involving $\mmod\Lambda$ itself; in particular it is known that its global dimension does not exceed $2$. In other words, a natural line of thought  posed by this formula is to do homological algebra within $\mmod(\mmod\Lambda)$ and then try to translate the results back for $\mmod\Lambda$. It should be highlighted that this exchange between  $\mmod(\mmod\Lambda)$ and $\mmod\Lambda$ led to a functorial approach to the fundamental concept of almost split sequences \cite{ARVI}. On the other hand, in \cite{Aus}, a functor $\theta:H(\Lambda)\lrt \mmod(\mmod\Lambda)$ has been constructed. The studies conducted in \cite{HM} reveal that this mapping respects certain types of almost split sequences.

\vspace{.1 cm}

Gluing together observations from previous paragraph motivates one to scrutinize further the morphism category $H(\Lambda)$ from representation-theoretic perspectives. Following these lines of thought, the Auslander-Reiten theory of the subcategory $\mathcal{P}(\Lambda)$ of $H(\Lambda)$ consisting of $\Lambda$-maps between projective modules has been investigated in \cite{Ba}. The idea of deriving almost split sequences and Auslander-Reiten translates in $H(\Lambda)$ from those in $\mmod\Lambda$ has also been followed e.g. in \cite{RS1}, \cite{XZZ}, and \cite{E}. Moreover, it might be interesting for the reader to draw attention to the interaction between the (mono)morphism category of $\Lambda$ and its (stable) Auslander algebra \cite{ARMatrix}, \cite{H}.

\vspace{.1 cm}

As declared above, this paper is devoted to study the morphism category $H(\Lambda)$ of an Artin algebra $\Lambda$, and this investigation concentrates mainly on the Auslander-Reiten theory of $H(\Lambda)$. Let us give an outline of what will be done throughout. In Section $2$, we compute the Auslander-Reiten translates of certain objects in $H(\Lambda)$. Based on what is needed in subsequent sections, the objects considered in this section include $(C\rt 0)$ and $(C\st{e}\rt I)$ where $e$ is the injective envelope of an indecomposable $\Lambda$-module $C$, and also non-projective indecomposable objects lying in $\mathcal{P}(\Lambda)$. It turns out that such computations might be done merely within the category $\mmod\Lambda$. In Section $3$ we recognize some almost split sequences in $H(\Lambda)$. In particular, we determine almost split sequences ending in an indecomposable object $(A\st{f}\rt B)$ where $f$ is a minimal right (resp. left) almost split map. Further, typical objects raised by projective covers $(P\st{p}\rt C)$ and injective envelopes $(C\st{e}\rt I)$ will be considered. Section $4$ looks at the middle terms of almost split sequences; part of this section considers those almost split sequences raised in Section $3$.

\vspace{.1 cm}

In the last section, restricting to self-injective algebras $\Lambda$, we apply our results to show that certain popular objects in $H(\Lambda)$ are $\tau_{\mathcal{H}}$-periodic; here $\tau_{\mathcal{H}}$ denotes the Auslander-Reiten translation in $H(\Lambda)$. Moreover, a link between morphism categories of finite representation type and Dynkin diagrams is established. More precisely, we will show, in certain circumstances, that the stable Auslander-Reiten quiver of $H(\Lambda)$, that is the quiver obtained from the Auslander-Reiten quiver of $H(\Lambda)$ be removing projective or injective vertices, is isomorphic to a quotient $\Z\Delta/G$ of the valued translation quiver $\Z\Delta$ of a Dynkin quiver $\Delta$. This result is obtained on the basis of a theorem by Liu \cite{L}. Regarding the fact that $H(\Lambda)$ is equivalent to the category of finitely generated right modules over $T_2(\Lambda)$, and that $T_2(\Lambda)$ is $1$-Gorenstein in this setting, we believe that such statements may serve as a motivation for extending the well-known results by Gabriel and Riedtmann \cite{R1, R2} concerning representation-finite hereditary and representation-finite self-injective algebras; see Remark \ref{RGR} for more details.

\vspace{.1 cm}

\noindent{\bf Notation and Conventions}.
Module categories and corresponding morphism categories are the only categories considered in this paper. However, in order to unify the expressions, we prefer to recall some basic materials in the category-theoretic language. For an object $X$ of a category $\CC$, we denote by $\rm{add}\mbox{-} X$ the category consisting of all direct summands of finite direct sums of copies of $X$. We say that a category is of finite representation type if there exists an object $X \in \CC$ such that $\CC=\rm{add}\mbox{-}X$.  An Artin algebra $\La$ is said to be of finite representation type (or representation-finite), if $\mmod \La$ is so. For a Krull-Schmidt category $\CC$, we denote by $\rm{ind}\mbox{-}\CC$ the subcategory of $\CC$ consisting of all indecomposable objects. To simplify the notation, we shall write $\rm{ind}\mbox{-}\Lambda$ to indicate finitely generated indecomposable $\Lambda$-modules.

\vspace{.1 cm}

Throughout the paper, with no potential ambiguity, we shall write $\mathcal{H}$ and $\mathcal{S}$ respectively instead of $H(\Lambda)$ and $S(\Lambda)$. If we regard the morphism $f:X\rt Y$ as an object in $\mathcal{H}$, depending on the context, we will use either of the symbols
$$\left(\begin{smallmatrix}
 X\\Y
 \end{smallmatrix}\right)_f, \ (X\st{f}\rt Y) \ \text{or} \ XY_f $$
where the last one will always be used in diagrams considering typographical issues.

\vspace{.1 cm}

We denote by $D_{\La}:\mmod \La\rt \mmod \La^{\rm{op}}$ the standard duality functor. Also the contravariant functor $\Hom_{\La}(-, \La):\mmod \La\rt \mmod \La^{\text{op}}$ is denoted $(-)^*_{\La}$. For the terminology and background on almost split morphisms we refer the reader to \cite{AuslanreitenSmalo, AS, L}. The Auslander-Reiten translation in $\mmod \La$ is denoted by  $\tau_{\La}$.  Moreover, $\nu_{\La}, \rm{Tr}_{\La}$ and $\Omega_{\La}^i$ stand respectively for the Nakayama functor,  the transpose and the $i$-th syzygy functor defined as usual by taking the kernels of the projective covers consecutively. In a dual manner, $\Omega^{-i}_{\La}$ is reserved to denote the $i$-th cosyzygy functor. The subscript $\Lambda$ will very often be omitted as this will cause no confusion.

\vspace{.1 cm}

Let us stress that indecomposable projective (resp. injective) objects in $\mathcal{H}$ are of the either forms $(0\rt P)$ or $(P\st{1}\rt P)$ (resp. $(I\rt 0)$ or $(I\st{1}\rt I)$) where $P$ (resp. $I$) is an indecomposable projective (resp. injective) $\Lambda$-module \cite{RS1}; this is freely used throughout the paper. Finally, the canonical inclusion and the canonical quotient maps are denoted respectively by $i$ and $p$ and the identity morphism of an object $C$ in a category $\CC$ is simply denoted by $1$.

\section{The translate of certain objects in $\mathcal{H}$}
The objective in this preliminary section is to compute the Auslander-Reiten translate of some certain objects in $\mathcal{H}$. Let us inaugurate it by adapting some necessary definitions from the Auslander-Reiten Theory of $\Lambda$. Throughout, the contravariant  functor
$(-)^*_\mathcal{H}:\mathcal{H}\rt {\mathcal{H}}^{{\rm op}} $ is defined for any object $\mathbb{X}=(X\st{f}\rt Y)$ by $\mathbb{X}_{\mathcal{H}}^*=(\Ker(f^*)\st{i}\rt Y^*)$; here $\mathcal{H}^{{\rm op}}$ stands for the morphism category of $\Lambda^{{\rm op}}$-modules. In particular, $(X\st{1}\rt X)^*_{\mathcal{H}}=(0 \rt X^*)$ and $(0\rt X)^*_{\mathcal{H}}=(X^*\st{1}\rt X^*)$. Further, the duality functor $D_{\mathcal{H}}:\mathcal{H}\rt \mathcal{H}^{\rm{op}}$ is recognized by sending $\mathbb{X}$ to  $D_{\mathcal{H}}(\mathbb{X})=(D(Y)\st{D(f)}\rt D(X))$; sometimes in the sequel the subscript $\mathcal{H}$ will be dropped.

\vspace{.1 cm}

It is well-known (e.g. from \cite[\S III]{AuslanreitenSmalo}) that there is an equivalence of categories $\Upsilon_{\La}:\mathcal{H}\rt \mmod T_2(\La)$.
Fortunately, $\Upsilon_{\La}$ is compatible with the aforementioned functors in the sense that one is given the commutative diagrams of functors

\[\xymatrix{\mathcal{H}\ar[r]^{(-)^*} \ar[d]_{\Upsilon_{\La}} & \mathcal{H}^{\rm{op}}\ar[d]^{\Upsilon_{\La^{\rm{op}}}} \\
\mmod T_2(\La) \ar[r]_ {(-)^*_{T_2(\La)}} & \mmod T_2(\La^{\rm{op}})}\]
and
\[\xymatrix{ \mathcal{H}\ar[r]^{D_{\mathcal{H}}} \ar[d]_{\Upsilon_{\La}} & \mathcal{H}^{\rm{op}}\ar[d]^{\Upsilon_{\La^{\rm{op}}}} \\ \mmod T_2(\La) \ar[r]_{D_{T_2(\La)}} & \mmod T_2(\La^{\rm{op}}) .}\]

With regard to this categorical equivalences and the aforementioned commutative diagrams, it is convincing to compute the Auslander-Reiten translation $\tau_{\mathcal{H}}$ of an object in terms of the duality and the transpose functors. However, in order to define the transpose ${\rm Tr}_\mathcal{H}$, we firstly need the following construction as a special case of a general procedure provided in \cite[Appendix A]{E}.

\begin{construction}\label{projective Cover}Let $\mathbb{X}=(M\st{f}\rt N)$ be an object in $\mathcal{H}$ and let $P\st{\alpha}\rt M$,  $Q\st{\beta} \rt \text{Coker}(f)$ be projective covers. By the  projectivity, there exists $\delta:Q\rt N$ such that $p \delta=\beta$, where $p:N\rt \text{Coker}(f)$ is the canonical epimorphism. Then by \cite[Theorem A1]{E} the morphism
	$$\left(\begin{smallmatrix}
\alpha\\ \left[\begin{smallmatrix}
f\alpha & \delta
\end{smallmatrix}\right]
\end{smallmatrix}\right):\left(\begin{smallmatrix}
P\\ P\oplus Q
\end{smallmatrix}\right)_{\left[\begin{smallmatrix}
	1\\ 0
	\end{smallmatrix}\right]}\rt \left(\begin{smallmatrix}M\\  N\end{smallmatrix}\right)_{f}$$	
defines the projective cover of $\mathbb{X}$ in $\mathcal{H}$.

\end{construction}
Therefore, symbolically we put $\rm{Tr}_{\mathcal{H}}(\mathbb{X})$  be defined as the cokernel of $(g)^*_{\mathcal{H}}$ where $\mathbb{P}_1\st{g}\rt \mathbb{P}_0\rt \mathbb{X}\rt 0$ is the minimal projective presentation of $\mathbb{X}$. Now, in view of the above remarks, it follows that  $\tau_{\mathcal{H}}(\mathbb{X}):=D_{\mathcal{H}} {\rm Tr}_{\mathcal{H}}(\mathbb{X})$ may be defined via the short exact sequence

$$0 \lrt \tau_{\mathcal{H}}(\mathbb{X})\lrt D_{\mathcal{H}}(\mathbb{P}_1)^*_{\mathcal{H}}\st{ D_{\mathcal{H}}(g)^*_{\mathcal{H}}}\lrt D_{\mathcal{H}}(\mathbb{P}_0)^*_{\mathcal{H}}\lrt D_{\mathcal{H}}(\mathbb{X})^*_{\mathcal{H}}\lrt 0$$
in $\mathcal{H}$.

\begin{proposition}\label{Prop 3.4}
Assume $C$ is an arbitrary $\Lambda$-module. Take the minimal projective presentation $P_1\st{\alpha}\rt P_0\rt C\rt 0$ of $C$. Then $\tau_{\mathcal{H}}(C\rt 0)\simeq (\nu P_1\st{\nu(\alpha)}\rt\nu P_0)$.
\end{proposition}
\begin{proof}
According to Construction \ref{projective Cover}, the exact sequence

$$\xymatrix@1{   {\left(\begin{smallmatrix} P_1\\ P_1\oplus P_0\end{smallmatrix}\right)}_{\left[\begin{smallmatrix} 1 \\ 0 			 \end{smallmatrix}\right]}
		\ar[rr]^-{\left(\begin{smallmatrix} \alpha \\ \left[\begin{smallmatrix} \alpha & 1
			\end{smallmatrix}\right]\end{smallmatrix}\right)}
		& & {\left(\begin{smallmatrix}P_0\\ P_0\end{smallmatrix}\right)}_{1}\ar[rr]^-{\left(\begin{smallmatrix} g \\ 0\end{smallmatrix}\right)}& &
		{\left(\begin{smallmatrix}C \\ 0\end{smallmatrix}\right)}_{0}\ar[r]& 0 \ \  \ \ \ \  \dagger}   $$
in $\mathcal{H}$ defines the minimal projective presentation of $(C\rt 0)$. Based on what we said above, the transpose of the object $(C\rt 0)$ in $\mathcal{H}$ is then isomorphic to $(P^*_0\st{\alpha^*}\rt P^*_1)$ as depicted in the diagram
	
	 	\[\xymatrix{0 \ar[r] \ar[d] & P^*_0 \ar[r]^1 \ar[d]^{\left[\begin{smallmatrix} 1 \\ 0
	 			\end{smallmatrix}\right]} & P^*_0  \ar[r] \ar[d]^{\alpha^*} & 0 \\ 	  P^*_0 \ar[r]^>>>>{\left[\begin{smallmatrix} 1 \\ \alpha^*
	 		\end{smallmatrix}\right]} &P^*_0\oplus P^*_1 \ar[r]^>>>>>>{\left[\begin{smallmatrix} \alpha^*& -1
	 		\end{smallmatrix}\right]} & P^*_1 \ar[r]  & 0.}\]
Hence one infers the desired result by applying the duality functor $D_{\mathcal{H}}$.
\end{proof}

We restrict to self-injective algebras to compute the Auslander-Reiten translate of injective envelopes as objects in $\mathcal{H}$.

\begin{proposition}\label{Prop 3.5}
Assume $\La$ is a self-injective algebra and $C$ is a $\Lambda$-module without projective direct summands. Consider the short exact sequence $0\rt C\st{e}\rt I\rt \Omega^{-1}(C)\rt 0$ in which $e$ is the injective envelope. Then $\tau_{\mathcal{H}}(C\st{e}\rt I)\simeq (\nu P_1\st{\nu(\pi)}\rt\nu\Omega(C))$, where the exact sequence $0 \rt \Omega^2(C)\rt P_1\st{\pi}\rt \Omega(C)\rt 0$ defines the projective cover of $\Omega(C)$ . In particular, if $p:P\rt\tau\Omega^{-1}(C)$ is the projective cover, then $\tau_{\mathcal{H}}(C\st{e}\rt I)\simeq (P\st{p}\rt \tau\Omega^{-1}(C))$.	
\end{proposition}
\begin{proof}
Let $P_0\st{c}\rt C$ be the projective cover. Then, by the Horseshoe Lemma, the exact sequence
$0 \rt \Omega^2(C)\rt P_1\rt P_0\st{c}\rt C\rt 0$ induces the commutative diagram

$$\xymatrix{& 0 \ar[d] & 0 \ar[d] & 0 \ar[d]& 0\ar[d]& &\\
	0 \ar[r] & \Omega^2(C) \ar[d] \ar[r] & P_1 \ar[d]
	\ar[r] &P_0 \ar[d] \ar[r] & C\ar[r]\ar[d]^e&0\\
	0 \ar[r] & P_1\ar[r]\ar[d]^{\pi} & P_1\oplus P_0\ar[d]
	\ar[r] & P_0\oplus I  \ar[r]\ar[d] & I\ar[r]\ar[d]& 0\\
0\ar[r] &\Omega(C)\ar[r]\ar[d]&P_0\ar[r]^{ec}\ar[d]&I\ar[d] \ar[r]&\Omega^{-1}(C) \ar[r]\ar[d]&0 &
\\ &0&0&0&0&&
} 	 $$
with exact rows and columns.
Since $\Lambda$ is self-injective and $C$ has no projective direct summands, we deduce that $I$ is isomorphic to the projective cover of $\Omega^{-1}(C)$. Therefore, Construction \ref{projective Cover} yields that the upper half of this diagram defines the minimal projective presentation of  $(C\st{e}\rt I)$ as well. Hence, applying the functor $(-)^*_{\mathcal{H}}$ we see, as before, that ${\rm Tr}_\mathcal{H}(C\st{e}\rt I)\simeq (\Omega(C)^*\st{{\pi}^*}\rt  P_1^*)$. The first assertion comes up. The second part follows from  \cite[Theorem IV.8.5]{SY}, stating that $\tau\simeq \nu\Omega^2$ as endo-functors over the stable category of $\Lambda$-modules.
\end{proof}

The following proposition deals with the Auslander-Reiten translate of non-projective objects in $\mathcal{H}$ that are locally represented by projective modules. This expression will be fruitful in the sequel.

\begin{proposition}\label{Prop 3.6}
	Let $(P\st{f}\rt Q)$ be an indecomposable non-projective object in $\mathcal{H}$ with $P$ and $Q$ projective. Set $M=\text{Coker}(P\st{f}\rt Q)$. Then $\tau_{\mathcal{H}}(P\st{f}\rt Q)= (0 \rt \tau M)$.
\end{proposition}
\begin{proof}
Since $(P\st{f}\rt Q)$ is indecomposable and non-projective in $\mathcal{H}$, we see that $M$ is indecomposable and non-projective. In fact, $M$ is not projective since otherwise the non-zero object $(\Ker(f)\rt 0)$ would be a direct summand of $(P\st{f}\rt Q)$.  We then claim that the quotient map $p: Q\rt M$ determines the projective cover of $M$. For, take $S\st{\lambda}\rt M$ to be the projective cover. This induces a $\Lambda$-homomorphism $\alpha: S\rt Q$ with $p\alpha=\lambda$. Likewise, there exists another map $\gamma:Q\rt S$ satisfying $\lambda\gamma=p$. Hence $\lambda=\lambda(\gamma\alpha)$ and the composite $\gamma\alpha$ should then be an isomorphism. These give rise to the endomorphism $(1, \alpha\gamma)$ of $(P\st{f}\rt Q)$ which should be either nilpotent or an isomorphism according to indecomposability of $(P\st{f}\rt Q)$. However, it is clearly not nilpotent; so $\alpha\gamma$ should be an isomorphism, indicating that $p: Q\rt M$ is the projective cover.

Meanwhile, it follows in a similar manner that $P \rt {\rm Im} (f)$ is the projective cover as well. Therefore, the minimal projective presentation $P\st{f}\rt Q\rt M \rt 0 $ of $M$ in $\mmod \La$ comes up. Now, the indecomposability of $(P\st{f}\rt Q)$ besides uniqueness of minimal projective presentations up to isomorphism gives the indecomposability of $M$. Indeed, any decomposition of $M$ would lead in a decomposition of the indecomposable object  $(P\st{f}\rt Q)$ due to the uniqueness of minimal projective presentations.

\vspace{.1 cm}

Applying Construction \ref{projective Cover} then gives the commutative diagram
\[\xymatrix{0\ar[r]&0 \ar[r] \ar[d] & P\ar[r]^1 \ar[d]^{\left[\begin{smallmatrix} 1\\0
				\end{smallmatrix}\right]} & P  \ar[r] \ar[d]^f & 0 \\0 \ar[r]&
		P \ar[r]^{\left[\begin{smallmatrix} -1 \\ f
			\end{smallmatrix}\right]} &P\oplus Q \ar[r]^{\left[\begin{smallmatrix} f &1
			\end{smallmatrix}\right]} & Q \ar[r]  & 0}\]
that provides us with the minimal projective presentation of $(P\st{f}\rt Q)$ in $\mathcal{H}$. It therefore turns out as before that $\tau_{\mathcal{H}}(P\st{f}\rt Q)$ is determined via the commutative diagram

\[\xymatrix{0 \ar[r]&0\ar[r]\ar[d]& \nu P \ar[r]^{\left[\begin{smallmatrix} 1 \\ \nu(f)
				\end{smallmatrix}\right]} \ar[d]^1 & \nu P\oplus \nu Q  \ar[d]^{\left[\begin{smallmatrix} 0 & 1
			\end{smallmatrix}\right]}    \\
0 \ar[r] & \tau M\ar[r]&
		\nu P \ar[r]^{\nu(f)} &\nu Q.      }\]
This completes the proof.
\end{proof}

\section{Certain almost split sequences in $\mathcal{H}$} \label{Section 4}

In this section the almost split sequences in $\mathcal{H}$ terminating in (or starting from) some certain objects will be determined explicitly. In the route to this goal, we apply some statements from previous section.  The Auslander-Reiten translates determined by the following lemma  will be used in the sequel; this is quoted from \cite{MO}.

\begin{lemma}\label{Lemma 5.2}
	Assume  $\delta:  0 \rt  A\st{f} \rt B\st{g} \rt C \rt 0$ is  an almost split sequence in $\mmod \La.$ Then
	\begin{itemize}
		\item[$(1)$] The almost split sequence in $\mathcal{H}$ ending at $(0\rt C)$ is of the form
			$$\xymatrix@1{  0\ar[r] & {\left(\begin{smallmatrix} A\\ A\end{smallmatrix}\right)}_{1}
			\ar[rr]^-{\left(\begin{smallmatrix} 1 \\ f\end{smallmatrix}\right)}
			& & {\left(\begin{smallmatrix}A\\ B\end{smallmatrix}\right)}_{f}\ar[rr]^-{\left(\begin{smallmatrix} 0 \\ g\end{smallmatrix}\right)}& &
			{\left(\begin{smallmatrix}0\\ C\end{smallmatrix}\right)}_{0}\ar[r]& 0. } \ \    $$		
	
		\item [$(2)$] The almost split sequence in $\mathcal{H}$ ending at $(C\st{1}\rt C)$ is of the form
			$$\xymatrix@1{  0\ar[r] & {\left(\begin{smallmatrix} A\\ 0\end{smallmatrix}\right)}_{0}
			\ar[rr]^-{\left(\begin{smallmatrix} f \\ 0 \end{smallmatrix}\right)}
			& & {\left(\begin{smallmatrix} B\\ C\end{smallmatrix}\right)}_{g}\ar[rr]^-{\left(\begin{smallmatrix} g \\ 1\end{smallmatrix}\right)}& &
			{\left(\begin{smallmatrix}C \\ C\end{smallmatrix}\right)}_{1}\ar[r]& 0. } \ \    $$		
		\end{itemize}
\end{lemma}
\begin{proof}
See \cite[Proposition 3.1]{MO}.
\end{proof}

Let us start with the following kind of almost split sequences.

\begin{proposition}\label{Prop4.1}
	Let $\delta: \  0 \rt A \st{f}\rt B \st{g}\rt C \rt 0$ and $\delta': \ 0 \rt A'\st{f'}\rt B'\st{g'}\rt A\rt 0$ be  almost split sequences in $\mathcal{H}$. Then
	$$\xymatrix@1{0\ar[r] & {\left(\begin{smallmatrix} B'\\ A\end{smallmatrix}\right)}_{g'}
		\ar[rr]^-{\left(\begin{smallmatrix}\left[\begin{smallmatrix} g'\\1
			\end{smallmatrix}\right]\\\left[\begin{smallmatrix} 1\\ f
			\end{smallmatrix}\right]
			\end{smallmatrix}\right)}
		& & {\left(\begin{smallmatrix}A\\ A\end{smallmatrix}\right)}_{1}\oplus{\left(\begin{smallmatrix}B'\\
			B\end{smallmatrix}\right)}_{fg'}\ar[rr]^-{\left(\begin{smallmatrix}\left[\begin{smallmatrix} -1 &g'
			\end{smallmatrix}\right]\\\left[\begin{smallmatrix} -f &1
			\end{smallmatrix}\right]
			\end{smallmatrix}\right)}& &
		{\left(\begin{smallmatrix}A\\B\end{smallmatrix}\right)}_{f}\ar[r]& 0 },$$
	is an almost split sequence  in $\mathcal{H}$. Further, $\left(\begin{smallmatrix}B'\\
	B\end{smallmatrix}\right)_{fg'}$ is  an indecomposable object.
\end{proposition}	
\begin{proof}
Firstly, since $\delta$ and $\delta'$ are almost split, we deduce that the objects $(A\st{f}\rt B)$ and $(B'\st{g'}\rt A)$ are indecomposable.
We shall show that $\tau_{\mathcal{H}}(A\st{f}\rt B)\simeq (B'\st{g'}\rt A)$.  From Lemma \ref{Lemma 5.2} we have the almost split sequences

	$$\xymatrix@1{ 0\ar[r] & {\left(\begin{smallmatrix} A\\ A\end{smallmatrix}\right)}_{1}
		\ar[rr]^-{\left(\begin{smallmatrix} 1 \\ f \end{smallmatrix}\right)}
		& & {\left(\begin{smallmatrix}A\\ B\end{smallmatrix}\right)}_{f}\ar[rr]^-{\left(\begin{smallmatrix} 0 \\ g\end{smallmatrix}\right)}& &
		{\left(\begin{smallmatrix}0 \\ C\end{smallmatrix}\right)}_{0}\ar[r]& 0, } \ \ \ \text{and}   $$
	$$\xymatrix@1{ 0\ar[r] & {\left(\begin{smallmatrix} A'\\ 0\end{smallmatrix}\right)}_{0}
		\ar[rr]^-{\left(\begin{smallmatrix} f' \\ 0\end{smallmatrix}\right)}
		& & {\left(\begin{smallmatrix}B'\\ A\end{smallmatrix}\right)}_{g'}\ar[rr]^-{\left(\begin{smallmatrix} g' \\ 1\end{smallmatrix}\right)}& &
		{\left(\begin{smallmatrix}A\\A\end{smallmatrix}\right)}_{1}\ar[r]& 0 }$$
in $\mathcal{H}$. These give rise to the following mesh in the Auslander-Reiten quiver of $\mathcal{H}$.

		$$	\xymatrix@-5mm{		
		&[B'A_{g'}]\ar[dr]&&[AB_f]\ar[dr]\\
[A'0]\ar[ur]&&[AA_1]\ar[ur]\ar@{.>}[ll]&&\ar@{.>}[ll][0C]	}	$$		
In view of \cite[\S VII, Proposition 1.5]{AuslanreitenSmalo}, 	one infers the aforementioned assertion. Notice that the sequence in the statement does not split since otherwise, we obtain either $(A\st{f}\rt B)\simeq (A\st{1}\rt A)$ or $(B'\st{g'}\rt A)\simeq (A\st{1}\rt A)$, both of which are impossible because $\delta$ and $\delta'$ do not split.

\vspace{.2 cm}

Consequently, thanks to \cite[\S V, Proposition 2.2]{AuslanreitenSmalo}, it suffices to show that any non-isomorphism
$\left(\begin{smallmatrix}
\phi_1\\\phi_2
\end{smallmatrix}\right):\left(\begin{smallmatrix}
B'\\ A
\end{smallmatrix}\right)_{g'}\rt \left(\begin{smallmatrix}
B'\\A
\end{smallmatrix} \right)_{g'}$ factors over
$\left(\begin{smallmatrix}\left[\begin{smallmatrix} g'\\1
	\end{smallmatrix}\right]\\\left[\begin{smallmatrix} 1\\ f
	\end{smallmatrix}\right]
\end{smallmatrix}\right)$.
Since $\left(\begin{smallmatrix}
\phi_1\\\phi_2
\end{smallmatrix}\right)$ is a non-isomorphism, at least one of $\phi_1$ and  $\phi_2$ is so. If it happens that $\phi_2$ is a non-isomorphism, then since $\delta'$ is an almost split sequence, there is a $\Lambda$-map $s:A\rt B'$ such that $g's=\phi_2.$ But then $g'(\phi_1-sg')=0$ so that there is $r:B'\rt A'$ with $\phi_1=f'r+sg'$. Now it is straightforward to verify that
	
$$\left(\begin{smallmatrix}\left[\begin{smallmatrix}s & f'r\end{smallmatrix}\right]\\ \left[\begin{smallmatrix}\phi_2  & 0\end{smallmatrix}\right]\end{smallmatrix}\right):\left(\begin{smallmatrix}A\\ A\end{smallmatrix}\right)_{1}\oplus\left(\begin{smallmatrix}B'\\
B\end{smallmatrix}\right)_{fg'}\rt \left(\begin{smallmatrix}B'\\ A\end{smallmatrix}\right)_{g'}$$
defines a morphism in $\mathcal{H}$ and does the factorization job as well.


\vspace{.1 cm}

Now we show it never happens that $\phi_2$ is an isomorphism and $\phi_1$ is not.  Assume, to the contrary, that this is the case.	The morphism $\left(\begin{smallmatrix}
	\phi_1\\\phi_2
\end{smallmatrix}\right)$
then induces a commutative diagram
	
$$\xymatrix
{0\ar[r] & A'\ar[r]^{f'}\ar[d]^{\phi_3} & B'\ar[r]^{g'}\ar[d]^{\phi_1} & A\ar[r] \ar[d]^{\phi_2} & 0\\
	0\ar[r] & A'\ar[r]_{f'} &B' \ar[r]_{g'} & A\ar[r] & 0
}$$	
where  $\phi_3$ must be a non-isomorphism according to the hypothesis. Now since $\delta'$ is an almost split sequences, there is $\la_1:B'\rt A'$ such that $\la_1f'=\phi_3$. However, as in the above lines, one obtains a $\Lambda$-map $\la_2:A\rt B'$ satisfying $\phi_1=\la_2g'+f'\la_1$ and $g'\la_2=\phi_2$. Since $\phi_2$ is an isomorphism, it follows that $g'$ is a split epimorphism, which is a contradiction.

\vspace{.1 cm}

To prove the second assertion, we may clearly think of $\left(\begin{smallmatrix}B'\\
B\end{smallmatrix}\right)_{fg'}$ as $\left(\begin{smallmatrix}B'\\
B\end{smallmatrix}\right)_{fg'}= \left(\begin{smallmatrix}X\\
	Y\end{smallmatrix}\right)_{v}\oplus\mathbb{L}\oplus\mathbb{E}$
where  $\mathbb{L}$ and $\mathbb{E}$ are respectively isomorphic to a direct sum of indecomposable objects of the form
$\left(\begin{smallmatrix} L\\
	0\end{smallmatrix}\right)$ and
$\left(\begin{smallmatrix} 0\\
	E\end{smallmatrix}\right)$
and  $\left(\begin{smallmatrix}X\\
	Y\end{smallmatrix}\right)_{v}$
has no direct summand of the two other types. We claim that $\mathbb{L}$ and $\mathbb{E}$ are respectively injective and projective.	For take an indecomposable direct summand $\left(\begin{smallmatrix} L\\
	0\end{smallmatrix}\right)$
of $\mathbb{L}$ and suppose, to the contrary, that $L$ is non-injective. Hence there exists an almost split sequence $0\rt L\st{x}\rt M\st{y}\rt N\rt 0$. By Lemma \ref{Lemma 5.2},  the almost split sequence in $\mathcal{H}$ starting from $\left(\begin{smallmatrix}L\\
	0\end{smallmatrix}\right)$ is of the form

	$$\xymatrix@1{0\ar[r] & {\left(\begin{smallmatrix} L\\ 0\end{smallmatrix}\right)}
		\ar[rr]^-{\left(\begin{smallmatrix} x \\ 0\end{smallmatrix}\right)}
		& & {\left(\begin{smallmatrix}M\\ N\end{smallmatrix}\right)}_{y}\ar[rr]^-{\left(\begin{smallmatrix} y \\ 1\end{smallmatrix}\right)}& &
		{\left(\begin{smallmatrix}N\\ N\end{smallmatrix}\right)}_{1}\ar[r]& 0 }.$$
Then \cite[Proposition 3.3]{AR} implies that the object $\left(\begin{smallmatrix} A\\ B\end{smallmatrix}\right)_f$ is a direct summand of $\left(\begin{smallmatrix} M\\ N\end{smallmatrix}\right)_y$. But since the latter one is indecomposable, we get $\left(\begin{smallmatrix} A \\ B\end{smallmatrix}\right)_f\simeq \left(\begin{smallmatrix} M \\ N\end{smallmatrix}\right)_y$  which is absurd because $f$ is not  an isomorphism. Thus $\mathbb{L}$ should be injective; that $\mathbb{E}$ is projective is verified similarly. Now if $\mathbb{E}$  is non-zero, it admits a non-zero indecomposable direct summand $\left(\begin{smallmatrix} 0\\ E\end{smallmatrix}\right)$. But this is projective and one observes, using \cite[Proposition 3.1]{AR}, that $\left(\begin{smallmatrix} B' \\ A\end{smallmatrix}\right)_{g'}$ is a subobject of the radical $\left(\begin{smallmatrix} 0 \\ \text{rad}(E)\end{smallmatrix}\right)$ of $\left(\begin{smallmatrix} 0 \\ E\end{smallmatrix}\right)$. This contradiction shows that $\mathbb{E}=0$ and, likewise, $\mathbb{L}=0$. Therefore  $\left(\begin{smallmatrix}B'\\
B\end{smallmatrix}\right)_{fg'}= \left(\begin{smallmatrix}X\\
	Y\end{smallmatrix}\right)_{v}$ evidently gives rise to $A=\im(fg')= \im(v)$. Adding that $A$ is indecomposable, it follows that $\im(v)$, and accordingly, $\left(\begin{smallmatrix}X\\
Y\end{smallmatrix}\right)_{v}$ is indecomposable.
\end{proof}

Let us point out that the above proposition provides a generalized version of \cite[Theorem 3.3]{E}.  As the projectivity of the first term $``A"$ is of particular interest to us, we record the following special case of Proposition \ref{Prop4.1}.


\begin{proposition}\label{Prop 4.3}
	Let $A$ be a projective $\Lambda$-module and $\delta: \  0 \rt A \st{f}\rt B \st{g}\rt C \rt 0$ be an almost split sequence in $\mmod \La$. Then
	$$\xymatrix@1{0\ar[r] & {\left(\begin{smallmatrix} \text{rad}(A)\\ A\end{smallmatrix}\right)}_{i}
		\ar[rr]^-{\left(\begin{smallmatrix}\left[\begin{smallmatrix} i\\1
			\end{smallmatrix}\right]\\\left[\begin{smallmatrix} 1\\ f
			\end{smallmatrix}\right]
			\end{smallmatrix}\right)}
		& & {\left(\begin{smallmatrix}A\\ A\end{smallmatrix}\right)}_{1}\oplus{\left(\begin{smallmatrix}\text{rad}(A)\\
			B\end{smallmatrix}\right)}_{fi}\ar[rr]^-{\left(\begin{smallmatrix}\left[\begin{smallmatrix} -1 & i
			\end{smallmatrix}\right]\\\left[\begin{smallmatrix} -f &  1
			\end{smallmatrix}\right]
			\end{smallmatrix}\right)}& &
		{\left(\begin{smallmatrix}A\\B\end{smallmatrix}\right)}_{f}\ar[r]& 0 },$$
	is an almost split sequence in $\mathcal{H}$.
\end{proposition}

\begin{proof}
By Lemma \ref{Lemma 5.2} we have the almost split sequence
		$$\xymatrix@1{ 0\ar[r] & {\left(\begin{smallmatrix} A\\ A\end{smallmatrix}\right)}_{1}
		\ar[rr]^-{\left(\begin{smallmatrix} 1 \\ f \end{smallmatrix}\right)}
		& & {\left(\begin{smallmatrix}A\\ B\end{smallmatrix}\right)}_{f}\ar[rr]^-{\left(\begin{smallmatrix} 0 \\ g\end{smallmatrix}\right)}& &
		{\left(\begin{smallmatrix}0 \\ C\end{smallmatrix}\right)}_{0}\ar[r]& 0 }   $$
in $\mathcal{H}$. Combined to \cite[Proposition VII.1.5]{AuslanreitenSmalo}, this leads to an irreducible map $\tau_{\mathcal{H}}{\left(\begin{smallmatrix} A\\ B\end{smallmatrix}\right)}_{f}\lrt \left(\begin{smallmatrix} A\\ A\end{smallmatrix}\right)_{1}$. On the other hand,  \cite[Lemma 1.3]{RS1} says that
$\xymatrix@1{   {\left(\begin{smallmatrix} \rm{ rad}(A)\\ A\end{smallmatrix}\right)}_{i}
			\ar[r]^-{\left(\begin{smallmatrix} i \\ 1\end{smallmatrix}\right)}
			&  {\left(\begin{smallmatrix}A\\ A\end{smallmatrix}\right)}_{1} }    $
is a minimal right almost split map in $\CH$. Hence $\tau_{\mathcal{H}}{\left(\begin{smallmatrix} A\\ B\end{smallmatrix}\right)}_{f}$ is a direct summand of ${\left(\begin{smallmatrix} \rm{ rad}(A)\\ A\end{smallmatrix}\right)}_{i}$.  But since $A$ is indecomposable, one deduces that ${\left(\begin{smallmatrix} \rm{ rad}(A)\\ A\end{smallmatrix}\right)}_{i}$ is also indecomposable. Therefore, we conclude that $\tau_{\mathcal{H}}{\left(\begin{smallmatrix} A\\ B\end{smallmatrix}\right)}_{f}\simeq \left(\begin{smallmatrix} \rm{rad}(A)\\ A\end{smallmatrix}\right)_{i}$. Now the rest of the proof goes ahead as in Proposition \ref{Prop4.1}.
\end{proof}

\remark  We would like to point out here that the sequence obtained in Proposition \ref{Prop 4.3} is also almost split in $\CS$ as $(\text{rad}(A)\st{i}\rt A)$ lies inside $\CS$.

\begin{proposition}\label{Prop 4.4}
	Assume $A$ is  an indecomposable non-projective $\La$-module and let $\eta: 0 \rt C \st{f}\rt B\st{g}\rt A\rt 0$ be an almost split sequence in $\mmod \La$. Let $P\st{p}\rt A$ and  $e:C\rt I$ be respectively the projective cover of $A$ and  the injective envelope  of $C$. Then the commutative diagram
        \[\xymatrix{0\ar[r]&C \ar[r]^{s} \ar@{=}[d] & Z\ar[r]^l \ar[d]^h 				 & P  \ar[r] \ar[d]^p & 0 \\0 \ar[r]&                     
        		C \ar[r]^{f}\ar[d]^e                
         & B \ar[r]^{g}\ar[d]^d       
        			 & A \ar[r]\ar@{=}[d]  & 0  \\0\ar[r] &I\ar[r]^u&X\ar[r]^v& A\ar[r]& 0
        }\]

with exact rows obtained by taking the pull-back and the push-out of $\eta$ along $p$ and $e$, induces the almost split sequence

 $$\xymatrix@1{\delta: \ \ 0\ar[r] & {\left(\begin{smallmatrix} C\\ I\end{smallmatrix}\right)}_{e}
	\ar[rr]^-{\left(\begin{smallmatrix} s \\ u\end{smallmatrix}\right)}
	& & {\left(\begin{smallmatrix}Z\\ X\end{smallmatrix}\right)}_{dh}\ar[rr]^-{\left(\begin{smallmatrix} l \\ v\end{smallmatrix}\right)}& &
	{\left(\begin{smallmatrix}P\\A\end{smallmatrix}\right)}_{p}\ar[r]& 0 }$$
in $\mathcal{H}$ ending at $\mathbb{X}=(P\st{p}\rt A)$.

\end{proposition}
\begin{proof}
First, we prove that $\delta$ does not split. For, suppose to the contrary that there exists $\left(\begin{smallmatrix}r_1\\ r_2\end{smallmatrix}\right):\left(\begin{smallmatrix}P\\ A\end{smallmatrix}\right)_p\rt \left(\begin{smallmatrix}Z\\ X\end{smallmatrix}\right)_{dh}$ such that $\left(\begin{smallmatrix}l\\ v\end{smallmatrix}\right)\left(\begin{smallmatrix}r_1\\ r_2\end{smallmatrix}\right)=\left(\begin{smallmatrix}1\\ 1\end{smallmatrix}\right).$
Let $q:X\rt \text{Coker}(d)$ be the canonical quotient map. We argue that $qr_2=0$. Indeed, for $a \in A$, there is $a'\in P$ with $p(a')=a$; the commutativity condition imposed by  the map $\left(\begin{smallmatrix}r_1\\ r_2\end{smallmatrix}\right)$  then implies $dhr_1(a')=r_2p(a')=r_2(a)$, and the claim follows. Hence there exists $w:A\rt B$ with $dw=r_2$.  But then $gw=vdw=vr_2=1_A$, meaning that $\eta$ splits. This contradiction reveals that $\delta$ does not split.

\vspace{.1 cm}

One observes easily that the end terms of $\delta$ are indecomposable. Let now $\left(\begin{smallmatrix}
\phi_1\\\phi_2
\end{smallmatrix}\right):\left(\begin{smallmatrix}
C\\ I
\end{smallmatrix}\right)_e\rt \left(\begin{smallmatrix}
V\\W
\end{smallmatrix} \right)_t$ be a non-section. It induces the map
	
$$\left(\begin{smallmatrix}
j\phi_1\\\phi_2
\end{smallmatrix}\right):\left(\begin{smallmatrix}
C\\ I
\end{smallmatrix}\right)_e\rt \left(\begin{smallmatrix}
\rm{Im}(t)\\W
\end{smallmatrix} \right)_i,$$
in $\CS(\La)$, where $V\st{j}\rt \text{Im}(t) \st{i}\rt W $ is a mono-epi factorization of $t$. This is plainly again a non-section.
The lower half of the diagram yields the sequence
	$$\xymatrix@1{0\ar[r] & {\left(\begin{smallmatrix} C\\ I\end{smallmatrix}\right)}_{e}
	\ar[rr]^-{\left(\begin{smallmatrix} f \\ u\end{smallmatrix}\right)}
	& & {\left(\begin{smallmatrix}B\\ X\end{smallmatrix}\right)}_{d}\ar[rr]^-{\left(\begin{smallmatrix} g \\ v\end{smallmatrix}\right)}& &
	{\left(\begin{smallmatrix}A\\A\end{smallmatrix}\right)}_{1}\ar[r]& 0 }$$
which is almost split in $\CS$ by Lemma 6.3 of \cite{H}.
Therefore,   $\left(\begin{smallmatrix}
	j\phi_1\\\phi_2
\end{smallmatrix}\right)$ factors over $\left(\begin{smallmatrix}
f\\u
\end{smallmatrix}\right)$ via, say, $\left(\begin{smallmatrix}
\alpha_1\\\alpha_2
\end{smallmatrix}\right):\left(\begin{smallmatrix}
B\\ X
\end{smallmatrix}\right)_d\rt \left(\begin{smallmatrix}
\rm{Im}(t)\\W
\end{smallmatrix} \right)_i$. In particular, $\phi_2=\alpha_2 u$. Since $P$ is projective, there should exist on the one hand $\Lambda$-maps $n:Z\rt C$ and $l':P\rt Z$ with $ns=1_C$ and $ll'=1_P$ and, on the other hand, $q:P\rt V$ with $\alpha_1 h l'=j q$. Adding that $i \alpha_1=\alpha_2 d$,  it is straightforward to deduce that   $\left(\begin{smallmatrix}
\phi_1 n+q l\\\alpha_2
\end{smallmatrix}\right):\left(\begin{smallmatrix}
Z\\ X
\end{smallmatrix}\right)_{dh}\rt \left(\begin{smallmatrix}
V\\W
\end{smallmatrix} \right)_t$ is a morphism in $\mathcal{H}  $ which factors $\left(\begin{smallmatrix}
\phi_1\\\phi_2
\end{smallmatrix}\right)$ through $\left(\begin{smallmatrix}
s\\u
\end{smallmatrix}\right)$.
\end{proof}

We now draw attention to the almost split sequence that starts from an object in $\mathcal{H}$ of the form $(0\rt P)$, where $P$ is an indecomposable projective-injective $\Lambda$-module.

\begin{proposition}\label{Prop 4.5}
	 Let $P$ be an indecomposable  projective-injective $\Lambda$-module. Assume $p:Q\rt \text{soc}(P)$ is the projective cover of the socle $\text{soc}(P)$ of $P$. Then
		$$\xymatrix@1{ \eta: \ \ 0\ar[r] & {\left(\begin{smallmatrix} 0\\ P\end{smallmatrix}\right)}_{0}
			\ar[rr]^-{\left(\begin{smallmatrix} 0\\ 1\end{smallmatrix}\right)}
			& & {\left(\begin{smallmatrix}Q\\ P \end{smallmatrix}\right)}_{ip}\ar[rr]^-{\left(\begin{smallmatrix} 1 \\ 0\end{smallmatrix}\right)}& &
			{\left(\begin{smallmatrix}Q\\0\end{smallmatrix}\right)}_{0}\ar[r]& 0 }$$
is an almost split sequence in $\mathcal{H}$, where  $i:\text{soc}(P)\rt P$ is the inclusion.
\end{proposition}

\begin{proof}
Since $\left(\begin{smallmatrix}0\\ P \end{smallmatrix}\right)_0$ is indecomposable non-injective, there exists an almost split sequence

	$$\xymatrix@1{ \epsilon: \ \ 0\ar[r] & {\left(\begin{smallmatrix} 0\\ P\end{smallmatrix}\right)}_{0}
	\ar[rr]^-{\left(\begin{smallmatrix} 0\\ g\end{smallmatrix}\right)}
	& & {\left(\begin{smallmatrix}X\\ Y \end{smallmatrix}\right)}_{f}\ar[rr]^-{\left(\begin{smallmatrix} r \\ s\end{smallmatrix}\right)}& &
	{\left(\begin{smallmatrix}V\\W\end{smallmatrix}\right)}_{d}\ar[r]& 0 }$$
in $\mathcal{H}$. We will show that the above sequence is isomorphic to $\eta$, whence the result. Note at the first pace that ${\left(\begin{smallmatrix}V\\W\end{smallmatrix}\right)}_{d}\simeq \tau^{-1}_{\mathcal{H}}{\left(\begin{smallmatrix} 0\\ P\end{smallmatrix}\right)}$. However, Construction \ref{projective Cover} besides projective-injectivity of $P$ yields that

$$\xymatrix@1{ \ \ 0\ar[r] & {\left(\begin{smallmatrix} 0\\ D(P)\end{smallmatrix}\right)}
			\ar[rr]^-{\left(\begin{smallmatrix} 0\\ 1\end{smallmatrix}\right)}
			& & {\left(\begin{smallmatrix}D(P)\\ D(P) \end{smallmatrix}\right)}_{1}\ar[rr]^-{\left(\begin{smallmatrix} 1 \\ 0\end{smallmatrix}\right)}& &
			{\left(\begin{smallmatrix}D(P)\\0\end{smallmatrix}\right)}\ar[r]& 0 }$$
is the minimal projective presentation of $D_{\mathcal{H}}\left(\begin{smallmatrix}0\\P \end{smallmatrix}\right)=\left(\begin{smallmatrix}D(P)\\ 0 \end{smallmatrix}\right)$. Hence an easy computation reveals that ${\left(\begin{smallmatrix}V\\W\end{smallmatrix}\right)}_{d}\simeq\tau^{-1}_{\mathcal{H}}{\left(\begin{smallmatrix} 0\\ P\end{smallmatrix}\right)}={\left(\begin{smallmatrix} \nu^{-1} P\\0\end{smallmatrix}\right)}$. In particular $W=0$ and $V\simeq\nu^{-1} P$ is projective indecomposable.
Further, since then $s=0$ and $r$ is an isomorphism, we may assume that $\epsilon$ has the following form
	$$\xymatrix@1{ \epsilon: \ \ 0\ar[r] & {\left(\begin{smallmatrix} 0\\ P\end{smallmatrix}\right)}
	\ar[rr]^-{\left(\begin{smallmatrix} 0\\ 1\end{smallmatrix}\right)}
	& & {\left(\begin{smallmatrix}\nu^{-1} P\\ P \end{smallmatrix}\right)}_{f}\ar[rr]^-{\left(\begin{smallmatrix} 1 \\ 0\end{smallmatrix}\right)}& &
	{\left(\begin{smallmatrix}\nu^{-1} P\\0\end{smallmatrix}\right)}_{0}\ar[r]& 0 }.$$

It remains to show that $Q\simeq \nu^{-1} P$. Clearly $f\neq 0$ as $\epsilon$ does not split. Denote by $j$ the inclusion $\text{rad}(\nu^{-1} P)\rt \nu^{-1} P$. Then the map  $\left(\begin{smallmatrix}
j\\ 0
\end{smallmatrix}\right):\left(\begin{smallmatrix}
\text{rad}(\nu^{-1} P)\\ 0
\end{smallmatrix}\right)\rt \left(\begin{smallmatrix}
\nu^{-1} P\\0
\end{smallmatrix} \right)$ factors through the epimorphism $\left(\begin{smallmatrix} 1 \\ 0\end{smallmatrix}\right)$ sitting in $\epsilon$ as it is obviously a non-retraction. Hence there is a morphism $\left(\begin{smallmatrix}
j\\ 0
\end{smallmatrix}\right):\left(\begin{smallmatrix}
\text{rad}(\nu^{-1} P)\\ 0
\end{smallmatrix}\right)\rt  \left(\begin{smallmatrix}
\nu^{-1} P\\P
\end{smallmatrix} \right)_f$. The upshot is that $fj=0$, and then there exists $x:\nu^{-1} P/\text{rad}(\nu^{-1} P)\rt P$ commuting the diagram
	$$\xymatrix{		
	\text{rad}(\nu^{-1} P) \ar[r]^{j} & \nu^{-1} P
	\ar[d]^f \ar[r]^<<<<<<{q} & \nu^{-1} P/\text{rad}(\nu^{-1} P) \ar[dl]^x  \\
	 & P
	 &   }
$$
where $q$ is the canonical quotient map. Since the socle $\text{soc}(P)$ of $P$ is simple and $f\neq 0$, the induced map $x:\nu^{-1} P/\text{rad}(\nu^{-1} P)\rt \text{soc}(P)$ is certainly an isomorphism that ascends to an isomorphism between projective covers; namely $\nu^{-1} P\simeq Q$ and the result follows.
\end{proof}

At this point, we want to summon some immediate consequences of the results so far obtained in Sections $2$ and $3$. Let us denote by $[(A\st{f}\rt B)]$ the iso-class of the object $(A\st{f}\rt B)$ in $\mathcal{H}$.

\begin{corollary}\label{Prop 5.4}
The functors $\tau_{\CH}$ and $\tau_{\CH}^{-1}$ induce mutually inverse equivalences		
\begin{itemize}
\item[$(a)$]$\mathcal{A}\approx\mathcal{B}$, where $$\mathcal{A}=\{[(P\st{f}\rt A)]\mid f \text{ is the projective cover of some indecomposable non-projective}\ A\in\mmod\La\},$$ $$\mathcal{B}=\{[(B\st{e}\rt I)]\mid  e \text{ is the injective envelope of some indecomposable non-injective}\ B\in\mmod\La  \}$$

\item[$(b)$]$\mathcal{C}\approx\mathcal{D}$, where $$\mathcal{C}=\{[(A\st{f}\rt B)]\mid \ f \text{ is a minimal left almost split map in}  \ \mmod\La, \ A \ \text{ non-injective}  \},$$ $$\mathcal{D}=\{[(C\st{g}\rt D)]\mid \ g \text{ is a minimal right almost split map in} \ \mmod \La, \ D \ \text{non-projective}\}$$

\item[$(c)$]and, $\mathcal{E}\approx\mathcal{F}$, where $$\mathcal{E}=\{[\mathbb{E}]\mid  \mathbb{E}=(P\st{f}\rt Q)  \ \text{is  indecomposable  non-projective in }\mathcal{H} \ \text{and} \ P, Q \ \text{ projective}\},$$   $$\mathcal{F}=\{[M] \mid  \ M \ \text{ is indecomposable non-projective in } \  \mmod \La \}.$$	
\end{itemize}		
\end{corollary}
\begin{proof}
  Parts $(a)$, $(b)$, and $(c)$ follow respectively from Propositions \ref{Prop 4.4}, \ref{Prop4.1}, and \ref{Prop 3.6}.
\end{proof}

We close the section by mentioning the following corollary which is well-known in the literature; see e.g., \cite[Proposition III.8.6]{SY}. Nonetheless, we  include it here to indicate how working in the morphism category $\mathcal{H}$ may affect, and ease, calculations related to Auslander-Reiten translates in $\mmod\La$.

\begin{corollary}
Let $P$ be an indecomposable projective-injective $\Lambda$-module. Then $\tau( P/\text{soc}(P))\simeq \text{rad}(P)$.
\end{corollary}
\begin{proof}
From \cite[Lemma 1.3]{RS1} we have  the minimal left almost split morphism $\left(\begin{smallmatrix}
1\\p
\end{smallmatrix}\right):\left(\begin{smallmatrix}
P\\ P
\end{smallmatrix}\right)_1\rt \left(\begin{smallmatrix}
P\\P/\text{soc}(P)
\end{smallmatrix} \right)_{p}$ and the  minimal right almost split morphism $\left(\begin{smallmatrix}
i\\1
\end{smallmatrix}\right):\left(\begin{smallmatrix}
\text{rad}(P)\\ P
\end{smallmatrix}\right)_i\rt \left(\begin{smallmatrix}
P\\P
\end{smallmatrix} \right)_{1}$ in $\mathcal{H}$.  Hence there exists an irreducible map $\tau_{\CH}\left(\begin{smallmatrix}
P\\ P/\text{soc}P
\end{smallmatrix}\right)_p\rt \left(\begin{smallmatrix}
P\\ P
\end{smallmatrix}\right)_1$. This shows that $\tau_{\CH}\left(\begin{smallmatrix}
P\\ P/\text{soc}P
\end{smallmatrix}\right)_p$ is isomorphic to a direct summand of $\left(\begin{smallmatrix}
\text{rad}(P)\\ P
\end{smallmatrix}\right)_i$. Adding that $\left(\begin{smallmatrix}
\text{rad}(P)\\ P
\end{smallmatrix}\right)_i$ is indecomposable by the hypothesis, one deduces that $\tau_{\CH}\left(\begin{smallmatrix}
P\\ P/\text{soc}P
\end{smallmatrix}\right)_p\simeq \left(\begin{smallmatrix}
\text{rad}(P)\\ P
\end{smallmatrix}\right)_i$. On the other hand, if $e:\tau (P/\text{soc}(P))\rt I$ is the injective envelope, then by Proposition \ref{Prop 4.4},  $\tau_{\CH}\left(\begin{smallmatrix}
P\\ P/\text{soc}P
\end{smallmatrix}\right)_p=\left(\begin{smallmatrix}
\tau (P/\text{soc}(P))\\  I
\end{smallmatrix}\right)_e$. This ends the proof.
\end{proof}

\section{The middle terms}\label{Subsection 4.2}

In this section we give some information about the middle terms of almost split sequences in $\mathcal{H}$. Some of these sequences have been already considered  in previous section. In particular, the form of indecomposable (projective) injective objects appearing in the middle terms will be discussed. Some of the results presented in this section will play crucial role in the next one.

\begin{proposition}\label{Prop 4.6}
Let $A$ be an indecomposable non-injective $\Lambda$-module. The middle term $\mathbb{B}$ of the almost split sequence in $\mathcal{H}$ ending at $(A\rt 0)$ is of the form  $\mathbb{B}=\mathbb{X}\oplus (I\rt 0)$, where $I$ is injective (possibly zero), and $\mathbb{X}$ is an indecomposable non-injective and non-projective object in $\mathcal{H}$. Further, if $A$ is not isomorphic to a direct summand of any module of the form $J/\text{soc}(J)$ for indecomposable injective $\Lambda$-modules $J$, then $I=0.$
\end{proposition}
\begin{proof}
Let
$0 \rt \mathbb{C}\rt \mathbb{B}\rt \mathbb{A}\rt 0$ be the almost split sequence in $\mathcal{H}$ ending at $\mathbb{A}=(A\rt 0)$. Since $A$ is non-injective, there is an almost split sequence $0 \rt A\st{f}\rt B\st{g}\rt C\rt 0$	in $\mmod \La$. Then,  by Lemma \ref{Lemma 5.2}, one has the almost split sequence
$$\xymatrix@1{   0\ar[r] & {\left(\begin{smallmatrix} A\\ 0\end{smallmatrix}\right)}_{0}
	\ar[rr]^-{\left(\begin{smallmatrix} f\\ 0\end{smallmatrix}\right)}
	& & {\left(\begin{smallmatrix}B\\ C \end{smallmatrix}\right)}_{g}\ar[rr]^-{\left(\begin{smallmatrix} g \\ 1\end{smallmatrix}\right)}& &
	{\left(\begin{smallmatrix}C\\C\end{smallmatrix}\right)}_{1}\ar[r]& 0 }$$
in $\mathcal{H}$ with  indecomposable middle term $(B\st{g}\rt C)$.
Evidently $(B\st{g}\rt C)$ is non-projective so that, by \cite[Proposition VII.1.5]{AuslanreitenSmalo},  there is an irreducible map	
$\left(\begin{smallmatrix}
\phi_1\\\phi_2
\end{smallmatrix}\right):\tau_{\mathcal{H}}\left(\begin{smallmatrix}
B\\ C
\end{smallmatrix}\right)_g\rt \left(\begin{smallmatrix}
A\\ 0
\end{smallmatrix} \right)_0$. Set $\mathbb{X}=\tau_{\mathcal{H}}\left(\begin{smallmatrix}
B\\ C
\end{smallmatrix}\right)_g$. Then $\mathbb{X}$ is necessarily non-injective and isomorphic to a direct summand of $\mathbb{B}$. Hence, for some $\mathbb{X'}$, we have $\mathbb{B}=\mathbb{X}\oplus \mathbb{X}'$. If $\mathbb{X}'=0$, the proof is complete. Else, if $\mathbb{X}'\neq 0$, then
$\mathbb{X}'$ is injective since otherwise the indecomposability of $(B\st{g}\rt C)$ will be rejected. Hence we may put $\mathbb{X}'=(I\rt 0)\oplus(J\st{1}\rt J)$ for injective $\Lambda$-modules $I$ and $J$.  But $J=0$ since otherwise by \cite[Lemma 1.3]{RS1},
$\left(\begin{smallmatrix}
1\\p
\end{smallmatrix}\right):\left(\begin{smallmatrix}
J\\ J
\end{smallmatrix}\right)_1\rt \left(\begin{smallmatrix}
J\\J/\text{soc}(J)
\end{smallmatrix} \right)_{p}$ is a minimal left almost split map in $\mathcal{H}$ ($p$ denotes the canonical quotient map). Hence $(A\rt 0)$ is isomorphic to a direct summand of $(J\st{p}\rt J/\text{soc}(J))$, making $A$ consequently into an injective module which is impossible. The first statement is now settled. The last assertion comes up by a similar argument; namely, if $I\neq 0$ then, again by \cite[Lemma 1.3]{RS1}, there is the minimal left almost split map $\left(\begin{smallmatrix}
I\\ 0
\end{smallmatrix}\right)\rt \left(\begin{smallmatrix}
I/\text{soc}(I)\\0
\end{smallmatrix} \right)$. Hence $A$ would be isomorphic to a direct summand of $I/\text{soc}(I)$, disaccording the hypothesis.
\end{proof}

In Proposition \ref{Prop 3.6}, we determined the Auslander-Reiten translation of indecomposable non-projective objects in $\mathcal{H}$ of the form $(P\rt Q)$. Below we try to recognize the middle term of the almost split sequence ending in such objects.

\begin{proposition}\label{Proposition 4.7}
Assume $\La$ is self-injective and $P, Q$ are projective $\Lambda$-modules. Assume further that $\mathbb{P}=(P\st{f}\rt Q)$ is an indecomposable non-projective object in $\mathcal{H}$. Then the middle term  $\mathbb{B}$ of the almost split sequence in $\mathcal{H}$ ending at $\mathbb{P}$ is of the form $\mathbb{B}=\mathbb{W}\oplus (0\rt V)$, where $V$ is a projective $\Lambda$-module (possibly zero) and $\mathbb{W}$ is indecomposable non-projective and non-injective.
\end{proposition}

\begin{proof}
Set $N=\text{Coker}(f)$ and $M:=\text{Coker}(\nu^{-1}P\st{\nu^{-1}(f)}\lrt \nu^{-1}Q)$. Since $\Lambda$ is self-injective, the sequence $\nu^{-1}P\st{\nu^{-1}(f)}\lrt \nu^{-1}Q\lrt M\lrt 0$ provides a projective presentation of $M$. Thus Proposition \ref{Prop 3.4}   yields that $\tau_{\mathcal{H}}\left(\begin{smallmatrix}
M\\
0
\end{smallmatrix}\right)\simeq \left(\begin{smallmatrix}
P\\
Q
\end{smallmatrix}\right)_{f}$. Therefore By Proposition \ref{Prop 4.6} one has the mesh diagram
\[
\xymatrix@C=2pc@R=1pc{
	{\phantom{III}}&&[\mathbb{X}]\ar[rdd]&& \\
	{\phantom{III}}&&[I_10]\ar[rd]&& \\
	&[PQ_f]\ar[ruu]\ar[ru]\ar[rdd]&\vdots&[M0]& \\
	{\phantom{III}}&&\vdots&& \\
	{\phantom{III}}&&[I_n0]\ar[ruu]&&
}
\]
in the Auslander-Reiten quiver of $\mathcal{H}$ which corresponds to the almost split sequence $$0\rt \mathbb{P}\rt \mathbb{X}\oplus \left(\begin{smallmatrix}
I\\
0
\end{smallmatrix}\right)\rt \left(\begin{smallmatrix}
M\\
0
\end{smallmatrix}\right)\rt 0.$$
Here $\mathbb{X}$ is indecomposable non-injective and the $I_i$ are indecomposable injective modules. Note that if $\mathbb{X}$ happens to be projective, then the aforementioned sequence shows that the projective dimension of $M$ is finite, so that it is projective as $\La$ is self-injective. This contradicts the non-projectivity of $\mathbb{P}$.  Therefore if, according to Proposition \ref{Prop 3.6},
$$0\lrt \left(\begin{smallmatrix}
0\\ \tau N
\end{smallmatrix}\right)\lrt \mathbb{B}\lrt\mathbb{P}\lrt 0$$ is the almost split sequence in $\mathcal{H}$ ending at $\mathbb{P}$, then $\mathbb{W}:=\tau_{\mathcal{H}}\mathbb{X}$ is a non-zero direct summand of $\mathbb{B}$. We claim that  the non-injective object $\mathbb{W}$ is in addition non-projective. Indeed, if this is not the case, then it is of the either forms $(P'\st{1}\rt P')$ or $(0\rt P')$ for an indecomposable projective $\Lambda$-module $P'$. However, the former does not occur as it is injective in $\mathcal{H}$.  Furthermore, the latter case would imply, by Proposition \ref{Prop 4.5}, that $\mathbb{X}=\tau_{\mathcal{H}}^{-1}(0\rt P')=(Q\rt 0)$ for some injective $\Lambda$-module $Q$. But this is injective; the claim thus follows.

\vspace{.1 cm}

Now let $ 0 \rt A\st{t}\rt B\st{r}\rt \tau N\rt 0$ be an almost split sequence in $\mmod \La$. Then Lemma \ref{Lemma 5.2} in conjunction with the data collected in previous paragraph gives the following part of the Auslander-Reiten quiver $\Gamma_{\CH}$

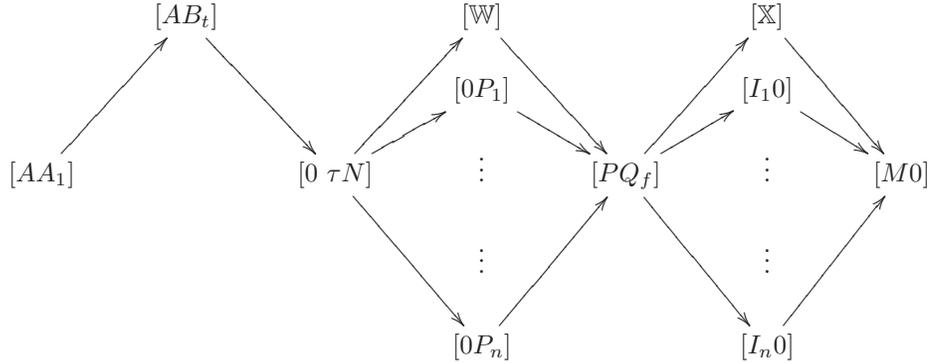
\[
\xymatrix@C=2pc@R=1pc{
&	[AB_t]\ar[rdd]&&[\mathbb{W}]\ar[rdd]&&[\mathbb{X}]\ar[rdd] && \\
&	&&[0P_1]\ar[rd]&&[I_10]\ar[rd]&& \\
[AA_1]\ar[ruu]&&[0 \,\,\tau N]\ar[ruu]\ar[ru]\ar[rdd]&\vdots&[PQ_f]\ar[ruu]\ar[ru]\ar[rdd]&\vdots &[M0]&\\
&&&\vdots&&\vdots&& \\
&	&&[0P_n]\ar[ruu]&&[I_n0]\ar[ruu]&&
}
\]
of $\mathcal{H}$, where the $P_i$ are projective. From the middle diagram we deduce $\mathbb{B}=\mathbb{W}\oplus (0\rt \oplus^n_{i=1} P_i)\oplus \mathbb{W}'$ where we add that  $\mathbb{W'}$ has no indecomposable projective direct summand of the form $(0\rt P')$ with $P'$ projective. We shall show $\mathbb{W}'=0$. One infers from the left-most mesh diagram that $\mathbb{W}'$ must be projective. Then its non-zero indecomposable direct summands should, according to the hypothesis, be only of the form  $(P''\st{1}\rt P'')$ for indecomposable projective $\Lambda$-modules $P''$. However, this may not happen since otherwise, as we did before,  $(0\rt \tau N)$ is isomorphic to a direct summand of  $\text{rad}(P''\st{1}\rt P'')=(\text{rad}(P'')\st{i}\rt P'')$. This forces $\tau N$  to be injective which is absurd.
\end{proof}


We now turn to study further almost split sequences of the type declared in Proposition \ref{Prop 4.4}.

\begin{proposition}\label{Prop 4.8}

Let $\mathbb{B}$ be the middle term of the almost split sequence whose existence was proved in Proposition \ref{Prop 4.4}. Then
	\begin{itemize}
		\item [$(1)$] $\mathbb{B}$ has no direct summand of the form $(0\rt Q)$ with $Q$  projective.
		\item [$(2)$] $\mathbb{B}$ has an indecomposable non-projective direct summand.
	\end{itemize}	
\end{proposition}

\begin{proof}
$(1)$ We keep the notation of \ref{Prop 4.4}. If there is such a projective indecomposable direct summand $(0\rt Q)$, then $(C\st{e}\rt I)$ is isomorphic to a direct summand of $\text{rad}(0\rt Q)=(0\rt \text{rad}(Q))$, which is impossible.

\vspace{.1 cm}

$(2)$  Assume, to the contrary, that $\mathbb{B}$ is projective. By the first part of the proposition, we can write $\mathbb{B}=(Z\st{dh}\rt X)\simeq \oplus (Q\st{1}\rt Q)$ for various projective $\Lambda$-modules $Q$. Hence $dh$ is an isomorphism and this makes $d$ into a retraction. Consequently, $e$ is an isomorphism so that $C$ is injective. But this is not consistent to the hypothesis therein.
\end{proof}

The following result shows that the converse to Proposition \ref{Prop 4.8} is also true under some mild assumption.

\begin{proposition}\label{prop 4.9}
Let $P$ be an indecomposable projective-injective $\Lambda$-module with indecomposable non-injective radical $\text{rad}(P)$. Then $(P\st{1}\rt P)$ appears as a direct summand of the middle term of an almost split sequence of the type declared in Proposition \ref{Prop 4.4}.
\end{proposition}

\begin{proof}
There is an almost split sequence $0 \rt \text{rad}(P)\rt A\rt B\rt 0$ in $\mmod \La$.	Moreover, the assumption implies that the inclusion $i:\text{rad}(P)\rt P$ is an injective envelope. Hence Proposition \ref{Prop 4.4} provides an almost split sequence
$$\xymatrix@1{\delta: \ \ 0\ar[r] & {\left(\begin{smallmatrix} \text{rad}(P)\\ P\end{smallmatrix}\right)}_{i}
	\ar[r]
	&  \mathbb{B}\ar[r] &
	\mathbb{K}\ar[r] & 0 }$$
in $\mathcal{H}$. Regarding the minimal right almost split map
$\left(\begin{smallmatrix}
i\\1
\end{smallmatrix}\right):\left(\begin{smallmatrix}
\text{rad}(P)\\ P
\end{smallmatrix}\right)_i\rt \left(\begin{smallmatrix}
P\\P
\end{smallmatrix} \right)_{1}$, we see that $(P\st{1}\rt P)$ is isomorphic to a direct summand of $\mathbb{B}$, as required.
\end{proof}


\begin{proposition}\label{PropositionThemiddle 4.5}
Assume $\La$ is self-injective and $C$ is a $\Lambda$-module. Let the injective envelope $(C\st{e} \rt I)$ of $C$ be an indecomposable non-projective object in $\mathcal{H}$. Let also $\mathbb{B}$ be the middle term of the almost split sequence in $\mathcal{H}$ ending at $(C\st{e}\rt I)$. Then $\mathbb{B}$ is not projective and has no direct summand of the form $(0\rt Q)$  with $Q$ projective.
\end{proposition}
\begin{proof}
The proof of the second claim goes ahead using the argument of Proposition \ref{Prop 4.8}. On the other hand, we know by Proposition \ref{Prop 3.5}  that  $\tau_{\mathcal{H}}\left(\begin{smallmatrix}
	C\\ I
	\end{smallmatrix}\right)_e= \left(\begin{smallmatrix}
	P\\ \tau\Omega^{-1}(C)
	\end{smallmatrix}\right)_p$, where $p$ is the projective cover of $\tau\Omega^{-1}(C)$. Now if $\mathbb{B}$ would be projective, then the aforementioned almost split sequence implies that $\Omega^{-1}(C)$ is projective, i.e. $C$ will then be injective which is denied because $(C\st{e}\rt I)$ was supposed to be non-projective in $\mathcal{H}$.
\end{proof}

\section{ Morphism categories of finite representation type}
The results of this section involve a self-injective and non-semisimple Artin algebra $\Lambda$. As an application of the results we obtained so far in previous parts of the paper, we will establish in this section a connection between
representation-finiteness of $\mathcal{H}$ and the so-called Dynkin diagrams.

\vspace{.1 cm}

For a module $M$ in $\mmod \La$, the $\tau$-orbit of $M$ is the set of all possible modules $\tau^n M$, $n \in \mathbb{Z}$. Then $M$ is called $\tau$-periodic if $\tau^m M\simeq M$ for some $m\geqslant 1.$  We refer to \cite{L} for the notion of the Auslander-Reiten quiver $\Gamma_{\CC}$ of a Krull-Schmidt category $\CC$ and relevant combinatorial background, particularly the notion of valued translation quivers. We suffice to recall that a valued translation quiver is represented by a triple $(\Gamma, \rho, \nu)$,  where $\Gamma$ is a quiver without multiple arrows, $\nu$ is a valuation of the arrows  , and $\rho$ is the translation. Also recall that a vertex $x$ in a valued translation quiver $(\Gamma, \rho, \nu)$ is called {\it stable} if $\rho^ix$ is defined for every integer  $i$.

\vspace{.1 cm}

For $i \in \mathbb{Z}$ different from $\pm 1$, we denote the $i$-th Auslander-Reiten translation of an object $\mathbb{X}$ in $\mathcal{H}$ by $\tau^{i}_{\mathcal{H}}\mathbb{X}$.

\begin{proposition}\label{Prop 5.1}
Assume $\La$ is self-injective of finite representation type. Any object $\mathbb{X}$ in $\mathcal{H}$ of either of the following types is $\tau_{\mathcal{H}}$-periodic.
\begin{itemize}
\item [$(1)$] $\mathbb{X}$ is either of $(0\rt C)$, $(C\st{1}\rt C)$, or $(C\rt 0)$ where $C$ is an indecomposable non-projective $\Lambda$-module.
\item [$(2)$] $\mathbb{X}$ is indecomposable non-projective-injective of the form $\mathbb{X}=(P\st{f}\rt Q)$ with $P, Q$ projective $\Lambda$-modules.
\item[$(3)$] $\mathbb{X}=(P\st{p}\rt C)$, where $P$ is the projective cover of an indecomposable non-projective $\Lambda$-module $C$.
\item [$(4)$] $\mathbb{X}=(C\st{e}\rt I)$, where $I$ is the injective envelope of an indecomposable non-projective $\Lambda$-module $C$.
\end{itemize}
\end{proposition}

\begin{proof}
$(1)$ We only treat the case where $\mathbb{X}=(0\rt C)$ since the proof of the other types is similar. Using Lemma \ref{Lemma 5.2} we get:
$$\tau_{\mathcal{H}}\left(\begin{smallmatrix}
0\\C
\end{smallmatrix}\right)=\left(\begin{smallmatrix}
\tau C\\ \tau C
\end{smallmatrix}\right)_1 \ \  \text{and} \ \ \tau^{2}_{\mathcal{H}}\left(\begin{smallmatrix}
0\\ C
\end{smallmatrix}\right)=\tau_{\mathcal{H}}\left(\begin{smallmatrix}
\tau C\\ \tau C
\end{smallmatrix}\right)_1=\left(\begin{smallmatrix}
\tau^2 C\\ 0
\end{smallmatrix}\right). $$
Moreover, by Propositions \ref{Prop 3.4} and  \ref{Prop 3.6},
$$\tau^{3}_{\mathcal{H}}\left(\begin{smallmatrix}
0\\ C
\end{smallmatrix}\right)=\tau_{\mathcal{H}}\left(\begin{smallmatrix}
\tau^2 C\\ 0
\end{smallmatrix}\right)=\left(\begin{smallmatrix}
\nu P_1\\ \nu P_0
\end{smallmatrix}\right)_{\nu(f)} \ \  \text{and} \ \ \tau^{4}_{\mathcal{H}}\left(\begin{smallmatrix}
0\\ C
\end{smallmatrix}\right)=\tau_{\mathcal{H}}\left(\begin{smallmatrix}
\nu P_1\\ \nu P_0
\end{smallmatrix}\right)_{\nu(f)}=\left(\begin{smallmatrix}
0\\ \tau\nu\tau^2 C
\end{smallmatrix}\right), $$
where $P_1\st{f}\rt P_0\rt \tau^2 C\rt 0$ is a minimal projective presentation. Note that to compute the fourth translation, the exact sequence $0 \rt \tau^3 C\rt \nu P_1\st{\nu(f)}\rt \nu P_0\rt \nu \tau^2 C\rt 0$ has been applied.

\vspace{.1 cm}

Set now $A=\tau\nu\tau^2$.
Then iterating the above process one obtains for positive integers $i=4m+k$ that
\begin{itemize}
	\item if $k=0$, $\tau^{i}_{\mathcal{H}}\left(\begin{smallmatrix}
	0\\ C
	\end{smallmatrix}\right)=\left(\begin{smallmatrix}
	0\\ A^mC
	\end{smallmatrix}\right),$
	\item if $k=1$, $\tau^{i}_{\mathcal{H}}\left(\begin{smallmatrix}
	0\\ C
	\end{smallmatrix}\right)=\left(\begin{smallmatrix}
	\tau A^m C\\ \tau A^m C
	\end{smallmatrix}\right)_{1},$
	\item if $k=2$, $\tau^{i}_{\mathcal{H}}\left(\begin{smallmatrix}
	0\\ C
	\end{smallmatrix}\right)=\left(\begin{smallmatrix}
	\tau^2 A^m C\\ 0
	\end{smallmatrix}\right),$
	\item if $k=3$, $\tau^{i}_{\mathcal{H}}\left(\begin{smallmatrix}
0\\ C
	\end{smallmatrix}\right)=\left(\begin{smallmatrix}
	\nu P^{m+1}_1\\ \nu P^{m+1}_0
	\end{smallmatrix}\right)_{\nu f^m},$
\end{itemize}
where $P^{m+1}_1\st{f^m}\rt P^{m+1}_0\rt \tau^2 A^m C\rt 0$ is a minimal projective presentation.

\vspace{.1 cm}

Furthermore, computing $\tau^{i}_{\mathcal{H}}$ for integers $i<0$, one obtains
$$\tau^{-1}_{\mathcal{H}}\left(\begin{smallmatrix}
0\\C
\end{smallmatrix}\right)=\left(\begin{smallmatrix}
Q_1\\ Q_0
\end{smallmatrix}\right)_g,$$
where $Q_1\st{g}\rt Q_0\rt \tau^{-1} C\rt 0$ is a minimal projective presentation, and

$$\tau^{-2}_{\mathcal{H}}\left(\begin{smallmatrix}
0\\ C
\end{smallmatrix}\right)=\tau^{-1}_{\mathcal{H}}\left(\begin{smallmatrix}
 Q_1\\ Q_0
\end{smallmatrix}\right)_g=\left(\begin{smallmatrix}
\nu^{-1}\tau^{-1}C\\ 0
\end{smallmatrix}\right), \quad\tau^{-3}_{\mathcal{H}}\left(\begin{smallmatrix}
0\\ C
\end{smallmatrix}\right)=\tau^{-1}_{\mathcal{H}}\left(\begin{smallmatrix}
\nu^{-1}\tau^{-1}C\\ 0
\end{smallmatrix}\right)=\left(\begin{smallmatrix}
\tau^{-1}\nu^{-1}\tau^{-1}C\\ \tau^{-1} \nu^{-1}\tau^{-1}C
\end{smallmatrix}\right)_1, $$
$$\tau^{-4}_{\mathcal{H}}\left(\begin{smallmatrix}
	0\\ C
\end{smallmatrix}\right)=\tau^{-1}_{\mathcal{H}}\left(\begin{smallmatrix}
	\tau^{-1}\nu^{-1}\tau^{-1}C\\ \tau^{-1}\nu^{-1}\tau^{-1}C
\end{smallmatrix}\right)_{1}=\left(\begin{smallmatrix}
	0\\ \tau^{-2} \nu^{-1}\tau^{-1} C
\end{smallmatrix}\right). $$
Hence by repeating this process we get for positive integers  $-i=4m+k$ that
\begin{itemize}
		\item if $k=0$, $\tau^{i}_{\mathcal{H}}\left(\begin{smallmatrix}
	0\\ C
	\end{smallmatrix}\right)=\left(\begin{smallmatrix}
0\\ A^{-m}C
	\end{smallmatrix}\right)$,
	
	\item if $k=1$, $\tau^{i}_{\mathcal{H}}\left(\begin{smallmatrix}
	0\\ C
	\end{smallmatrix}\right)=\left(\begin{smallmatrix}
	Q^m_1\\ Q^m_0
	\end{smallmatrix}\right)_{g^m}$,

	\item if $k=2$, $\tau^{i}_{\mathcal{H}}\left(\begin{smallmatrix}
	0\\ C
	\end{smallmatrix}\right)=\left(\begin{smallmatrix}
\nu^{-1}\tau^{-1}A^{-m}C\\ 0
	\end{smallmatrix}\right)$,
	
	\item if $k=3$, $\tau^{i}_{\mathcal{H}}\left(\begin{smallmatrix}
	0\\ C
	\end{smallmatrix}\right)=\left(\begin{smallmatrix}
\tau^{-1}\nu^{-1}\tau^{-1} A^{-m} C	\\ \tau^{-1}\nu^{-1}\tau^{-1} A^{-m} C
	\end{smallmatrix}\right)_{1}$,
\end{itemize}
where $Q^{m+1}_1\st{g^m}\rt Q^{m+1}_0\rt \tau^{-1} A^{-m}C\rt 0$ is a minimal projective presentation. Based on the above computations we  deduce that $(0\rt C)$ is $\tau_{\mathcal{H}}$-periodic if and only if $C$ is $A$-periodic in the sense that there is some $l>0$ with $A^lC\simeq C$. But the latter holds since the set $\{A^jC\}_{j \in \mathbb{Z}}$ contains only finitely many indecomposable modules up to isomorphism according to the hypothesis.

\vspace{.2 cm}

\noindent $(2)$ This follows by mimicking the argument above; so we skip it.

\vspace{.2 cm}

\noindent $(3)$ We compute the $i$-th Auslander-Reiten translation of $\mathbb{X}$ for $i>0$. Proposition \ref{Prop 4.4} yields
$$\tau_{\mathcal{H}}\left(\begin{smallmatrix}
P\\ C
\end{smallmatrix}\right)_p=\left(\begin{smallmatrix}
\tau C	\\ I
\end{smallmatrix}\right)_{e},$$
where $\tau C\st{e}\rt I$ is the injective envelope.
Proposition \ref{Prop 3.5}  then gives

$$\tau^{2}_{\mathcal{H}}\left(\begin{smallmatrix}
P\\ C
\end{smallmatrix}\right)_p=\tau_{\mathcal{H}}\left(\begin{smallmatrix}
\tau C\\ I
\end{smallmatrix}\right)_e=\left(\begin{smallmatrix}
Q	\\ \tau\Omega^{-1}\tau C
\end{smallmatrix}\right)_{q}, $$
where $Q\st{q}\rt \tau\Omega^{-1}\tau C$ is the projective cover. Set $B=\tau\Omega^{-1}\tau$. Then, iterating this process for positive integers $i=2m+k$ it is easy to verify that
\begin{itemize}
	\item if $k=0$, $\tau^{i}_{\mathcal{H}}\left(\begin{smallmatrix}
	P\\ C
	\end{smallmatrix}\right)_p=\left(\begin{smallmatrix}
	Q^m	\\ B^m C
	\end{smallmatrix}\right)_{q^m},$
	where $q^m: Q^m\rt B^mC$ is the projective cover.
	\item if  $k=1$, $\tau^{i}_{\mathcal{H}}\left(\begin{smallmatrix}
	P\\ C
	\end{smallmatrix}\right)_p=\left(\begin{smallmatrix}
\tau B^m C	\\ I^{m+1}
	\end{smallmatrix}\right)_{e^m},$
	where $e^m: \tau B^m C\rt I^{m+1}$ is the injective envelope.
\end{itemize}

\vspace{.2 cm}

Using the dual process, it follows for positive integers $-i=2m+k$ that
\begin{itemize}
		\item if  $k=0$, $\tau^{i}_{\mathcal{H}}\left(\begin{smallmatrix}
	P\\ C
	\end{smallmatrix}\right)_p=\left(\begin{smallmatrix}
	T^m	\\ B^{-m}C
	\end{smallmatrix}\right)_{r^m},$
	where $r^m:T^m\rt B^{-m}C$ is the projective cover.
	
	\item if $k=1$, $\tau^{i}_{\mathcal{H}}\left(\begin{smallmatrix}
	P\\ C
	\end{smallmatrix}\right)_p=\left(\begin{smallmatrix}
\Omega\tau^{-1}B^{-m}C	\\J^m
	\end{smallmatrix}\right)_{s^m},$
	where $s^m:\Omega\tau^{-1} B^{-m}C\rt J^m$ is the injective envelope.
	\end{itemize}
Note that there is an injection, at the level of objects, from $\text{ind}\mbox{-}\La$ to $\text{ind}\mbox{-}\rm{H}(\La)$ via taking projective covers (or injective envelopes). Hence, based on the above computations and the aforementioned injection, we deduce that $(P\st{p}\rt C)$ is  $\tau_{\mathcal{H}}$-periodic if and only if $C$ is $B$-periodic. However, this is the case following the assumptions of the proposition.

\vspace{.1 cm}

\noindent $(4)$ This is dual to $(3)$.
\end{proof}


Below, for the sake of simplicity, we do not distinguish between indecomposable objects and vertices of the Auslander-Reiten quiver $\Gamma_{\mathcal{H}}$ of $\mathcal{H}$. We denote by $\Gamma^s_{\mathcal{H}}$ the translation quiver obtained from $\Gamma_{\mathcal{H}}$ by removing all vertices corresponding to  objects which are either indecomposable projective  or indecomposable injective and the arrows attached to them. It should be pointed out that in the literature, the symbol $\Gamma^s$ is usually used to denote the translation quiver obtained by removing vertices that are simultaneously projective and injective.

\begin{lemma}\label{Lemma 6.5}
	Assume $\La$ is self-injective.
	\begin{itemize}
		\item [$(a)$]Then every vertex of $\Gamma^s_{\mathcal{H}}$ is stable.
		\item [$(b)$] Assume further that $\La$ is indecomposable and $\mathcal{H}$ is of finite representation type. Then the valued translation quiver $\Gamma^s_{\mathcal{H}}$ is connected.
		\end{itemize}
\end{lemma}
\begin{proof}
$(a)$	Assume,  to the contrary, that there is a vertex  $\mathbb{X}$ in $\Gamma^s_{\mathcal{H}}$ that  is not stable, i.e. for some $m \in \mathbb{Z}$, $\tau^ m _{\mathcal{H}}\mathbb{X}$ is not well-defined. Taking  $m>0$  does not harm the generality. Since $\tau^ m _{\mathcal{H}}\mathbb{X}$ is not defined, $\tau^ {m-1} _{\mathcal{H}}\mathbb{X}$ should be projective. Hence it will be of either forms $(0\rt P)$ or $(P\st{1}\rt P)$ for a projective indecomposable $\Lambda$-module $P$. But since $\Lambda$ is self-injective, an application of Proposition \ref{Prop 4.5} reveals that $\mathbb{X}$ has to be projective (possibly zero). This is impossible since by definition $\Gamma^s_{\mathcal{H}}$ does not contain such a vertex.

\vspace{.1 cm}

\noindent $(b)$ Assume $\La$ is indecomposable. Then,  it is not difficult to see that the $T_2$-extension $T_2(\Lambda)$ of $\Lambda$  is also indecomposable. Hence by \cite[\S VII, Theorem 2.1]{AuslanreitenSmalo} we observe that $\Gamma_{\mathcal{H}}$ has only one component.
Now take two vertices  $x$ and $y$ of $\Gamma^{s}_{\CH}.$ The connectedness of $\Gamma_{\CH}$ provides us with a walk $y=x_0\longleftrightarrow x_1 \longleftrightarrow \cdots \longleftrightarrow x_t=x $ where  $x_d\longleftrightarrow x_{d+1}$ means that there is either an arrow $x_d\rt x_{d+1}$ or, vice versa, $x_{d+1}\rt x_d$ in $\Gamma_{\CH}$.  If the $x_i$ are all non-projective and non-injective, there is nothing to prove. Assume there is $j$ such that $1\leq j\leq t-1$ and $x_j$ is projective or injective. Then the following cases might be distinguished:
\begin{itemize}
\item[$(1)$] $x_{j-1}\rt x_j \rt x_{j+1}$
\item[$(2)$]$x_{j-1}\leftarrow x_j \rt x_{j+1}$
\item[$(3)$]$x_{j-1} \leftarrow x_j \leftarrow x_{j+1}$
\item[$(4)$]$x_{j-1}\rt x_j\leftarrow x_{j+1}$
\end{itemize}
Evidently, we only need to treat the first two ones. Note that since every irreducible map is either an epimorphism or a monomorphism, the non-semisimplicity of $\Lambda$ implies that neither  $x_{j-1}$ nor $x_{j+1}$ can be projective.

\vspace{.1 cm}

Assume $(1)$ occurs. Then, clearly, $x_{j-1}=\tau_{\CH}\,x_{j+1}$ and one infers the existence of the almost split sequence $0 \rt x_{j-1} \rt \mathbb{X}\oplus x_{j}\rt x_{j+1}\rt 0$ in $\CH$.  By the hypothesis, $x_j$  can be of either types $(P\st{1}\rt P)$, $(P\rt 0)$, or $(0\rt P)$ for an indecomposable projective $\Lambda$-module $P$. Note that in the first case, by \cite[Lemma 1.3]{RS1}, $x_j$ is the  middle term of an almost split sequence that either starts from  $({\rm rad}(P)\st{i}\rt P)$ or terminates at $(P\st{p}\rt P/{\rm soc }(P))$.  Hence  one may respectively apply Propositions \ref{Prop 4.8}, \ref{Prop 4.6}, or \ref{Proposition 4.7} to these three possibilities to deduce that $\mathbb{X} $ has an indecomposable non-projective and non-injective  direct summand $\mathbb{X}'$. Therefore,  there is a walk $x_{j-1}\rt \mathbb{X}' \rt x_{j+1}$ in $\Gamma^s_{\mathcal{H}}$ which remedies the punctured walk.

\vspace{.1 cm}

Let now $(2)$ be the case. Setting $z=\tau_{\CH}x_{j+1}$, there exists a walk $z\rt x_j\rt x_{j+1}$ in $\Gamma_{\CH}$. Moreover, the former case gives us a walk $z \rt \mathbb{V}\rt x_{j+1}$ in $\Gamma^s_{\CH}$. Likewise,  one obtains another walk $x_{j-1} \leftarrow \mathbb{W} \leftarrow z$ in $\Gamma^s_{\CH}$. Together, these give rise to the walk $x_{j-1}\leftarrow \mathbb{W} \leftarrow z \rt \mathbb{V}\rt x_{j+1}$ in $\Gamma^s_{\CH}$.
\end{proof}

Now we are ready to state the main theorem of this section.

\begin{theorem}\label{Theorem 6.9}
Assume $\La$ is an indecomposable self-injective algebra. Then
	\begin{itemize}		
				\item [$(1)$] If $\mathcal{H}$ is of finite representation type, then $\Gamma^s_{\mathcal{H}}\simeq \mathbb{Z}\Delta/G$, where $\Delta$ is a Dynkin quiver and $G$ is an automorphism group of $\mathbb{Z}\Delta$.
		\item [$(2)$] If $\mathcal{H}$ is of infinite representation type, then every component of $\Gamma^s_{\mathcal{H}}$ containing an object of the form $(0\rt C)$  with $C$ indecomposable non-projective,  is a stable tube.
	\end{itemize}
\end{theorem}
\begin{proof}
According to Lemma \ref{Lemma 6.5},   $\Gamma^s_{\mathcal{H}}$ is  finite,  connected, and stable. Moreover,  any  object $(0\rt C)$, where $C$ is an indecomposable non-projective $\Lambda$-module, is $\tau_{\mathcal{H}}$-periodic by Proposition \ref{Prop 5.1}. Now both of the assertions follow from \cite[Theorem 5.5]{L}.
\end{proof}

\begin{remark}\label{RGR} The theorem above says that for $\Lambda$ indecomposable and self-injective, $\Gamma^s_{\mathcal{H}}\simeq \mathbb{Z}\Delta/G$, where $\Delta$ is a Dynkin quiver and $G$ is an automorphism group of $\mathbb{Z}\Delta$ provided that $T_2(\Lambda)$ is of finite representation type. Note that $\Lambda$ being self-injective implies that $T_2(\La)$ is 1-Gorenstein in the sense that injective dimension of the regular module is at most one. On the other hand, there are known classifications of the representation-finite hereditary as well as representation-finite self-injective algebras, mainly due respectively to Gabriel \cite{G} and Riedtmann \cite{R1,R2}, in terms of Dynkin diagrams; see also \cite{BLR}. As the class of 1-Gorenstein algebras clearly contains the class of hereditary as well as self-injective algebras, Theorem \ref{Theorem 6.9} suggests that one may expect to discover  a link to Dynkin diagrams for general representation-finite 1-Gorenstein algebras. It sounds more interesting when one recalls that there are several important classes of 1-Gorenstein algebras including the cluster-tilted algebras, 2-Calabi-Yao tilted algebras, or more generally, the endomorphism algebras of cluster tilting objects in triangulated categories, and also the class of 1-Gorenstein algebras defined by Geiss, Leclerc and Schr\"{o}er via quivers with relations associated with symmetrizable Cartan matrices \cite{GLS}.
\end{remark}

\vspace{.1 cm}

We illustrate with the following examples.

\example Let $k$ be a field and consider the $k$-algebra $\La=k[x]/(x^2)$. Put $S=k[x]/(x)$ be the simple $\La$-module. Let $p:\La\rt S$ be the canonical epimorphism,  $i:S\rt \La$ be the monomorphism defined by multiplication by $x$, and $f=ip$.  The Auslander-Reiten quiver $\Gamma_{\mathcal{H}}$ has been computed on page $227$ of \cite{AuslanreitenSmalo}.   It follows that $\Gamma^s_{\mathcal{H}}$ has the following shape.
 
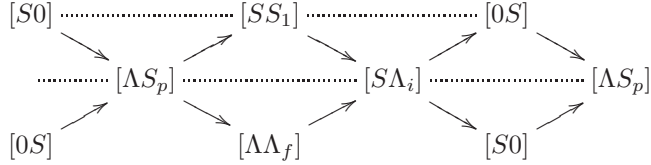
\[
\xymatrix  @R=0.3cm  @C=0.6cm {
	&&[S0]\ar[dr]\ar@{.}[rr]&	&\ar@{.}[rr][SS_1]\ar[dr]&&[0S]\ar[dr]&&&\\&&\ar@{.}[r]&
	[\La S_p]\ar[dr]\ar@{.}[rr]\ar[ru]&&[S\La_i]\ar[dr]\ar[ru]\ar@{.}[rr]&&[\La S_p]&&\\
	&&[0S]\ar[ur]&	&[\La\La_f]\ar[ru]&&[S0]\ar[ur]&&&}
\]
This shows that $\Gamma^s_{\mathcal{H}}=\mathbb{Z}\Delta/G$, where $\Delta$ is the quiver  $\bullet\rt \bullet \leftarrow \bullet$ and $G$ is generated by $\rho \tau^2_{\CH}$, where $\rho$ is the autoequivalence that permutes $[S0]$ and $[0S]$ and keeps $[\La S_p]$.

\vspace{.2 cm}

\begin{example}
 Let  $k$ be a field and $\La=k[x]/(x^n)$ for $n\geq 1$.
The Auslander-Reiten quiver of $\La$ looks like
 \[\xymatrix{
 	 k\ar@<2pt>[r]&k[x]/(x^2)\ar@<2pt>[r]\ar@<2pt>[l]&k[x]/(x^3)\ar@<2pt>[r]\ar@<2pt>[l]&\cdots\ar@<2pt>[r]\ar@<2pt>[l]&k[x]/(x^{n-1})\ar@<2pt>[r]\ar@<2pt>[l]&k[x]/(x^n)
 \ar@<2pt>[l]
 }\]
It follows that $\tau$ acts identically on indecomposable non-projective modules. In addition, as $\nu(\Lambda)\simeq \Lambda$ and $\nu$ induces an automorphism on the Auslander-Reiten quiver of $\La$, one infers that $\nu$ acts also identically over indecomposable objects as well. In view of the proof of Proposition \ref{Prop 5.1}, this yields that for $1\leq i\leq n-1$, the objects
 $$\left(\begin{smallmatrix}
0\\k[x]/(x^i)
\end{smallmatrix}\right)_0, \left(\begin{smallmatrix}
k[x]/(x^i)\\k[x]/(x^i)
\end{smallmatrix}\right)_1, \left(\begin{smallmatrix}
k[x]/(x^i)\\ 0
\end{smallmatrix}\right)_0 \ \ \text{and} \ \ \left(\begin{smallmatrix}
k[x]/(x^n)\\k[x]/(x^n)
\end{smallmatrix}\right)_h, $$
are all of $\tau_{\mathcal{H}}$-periodicity $4$; here $h$ is the composition of the  surjection  $k[x]/(x^n)\rt k[x]/(x^i)$ defined by $1+(x^n)\mapsto 1+(x^i) $ and the injection  $k[x]/(x^i)\rt k[x]/(x^n)$ given by $1+(x^i)\mapsto x^{n-i}+(x^n)$.
\end{example}

\begin{remark} Motivated by Proposition \ref{Prop 5.1}, for a self-injective Artin algebra $\La$, we consider  $A=\tau\nu\tau^2$  as an autoequivalence of the stable category $\underline{\text{mod}}\mbox{-}\La$ of $\Lambda$.
A fundamental property of the Nakayama functor $\nu$ is that it commutes with any auto-equivalence of $\underline{\text{mod}}\mbox{-}\La$; see e.g. \cite[Page 15]{DI}. Hence $A\simeq\nu\tau^3\simeq\tau^3\nu$.
Similarly, we are given an equivalence $B=\tau\Omega^{-1}\tau$ on $\underline{\text{mod}}\mbox{-}\La$. Since $\Omega$ defines an autoequivalence on the stable category, one has $B\simeq\Omega^{-1}\tau^2\simeq \tau^{2}\Omega^{-1}$ at the level of indecomposable non-projective modules. In the meanwhile, by \cite [\S VI, Theorem 8.5]{SY}, we know that $\tau\simeq \nu\Omega^2\simeq \Omega^2\nu$. This implies that, at the level of indecomposable non-projective modules, $A\simeq \nu^4\Omega^6\simeq \Omega^6\nu^4$ whereas $B\simeq \Omega^3\nu\simeq \nu\Omega^3$.
In particular, if $\Lambda$ is symmetric, then $A\simeq \Omega^6$ and $B\simeq  \Omega^3$.
\end{remark}

We conclude the paper by providing an application of the results in this section. Assume $\La$ is self-injective and $M$ is an indecomposable non-projective $\Lambda$-module. Denote by $[M]_A$ (resp. $[M]_B$) the $A$-orbit (resp. the $B$-orbit) of $M$.   Let also $\Gamma_{\CH}(M)$ (resp. $\Gamma'_{\CH}(M)$) be the unique component of the  Auslander-Reiten quiver $\Gamma_{\CH}$ containing the vertex $\left(\begin{smallmatrix}0\\ M \end{smallmatrix}\right)_0$ (resp. $\left(\begin{smallmatrix}P\\ M \end{smallmatrix}\right)_p$ where $p:P\rt M$ is the projective cover). Moreover, let
$$\mathcal{T}=\{\Gamma_{\CH}(M)\mid M \in \text{ind}\mbox{-}\La\},\ \ \mathcal{T}'=\{\Gamma'_{\CH}(M)\mid M \in \text{ind}\mbox{-}\La\}$$
$$\mathcal{E}=\{[M]_{A}\mid M \in \text{ind}\mbox{-}\La \} \ \text{and} \ \ \mathcal{E}'=\{[M]_{B}\mid M \in \text{ind}\mbox{-}\La \}.$$
Define the map $\delta:\mathcal{E}\rt \mathcal{T}$ (resp. $\beta:\mathcal{E}'\rt \mathcal{T}'$) by sending $[M]_{A}$ to $\Gamma_{\CH}(M)$ (resp. $[M]_{B}$ to $\Gamma'_{\CH}(M)$). These maps are well-defined. Indeed, if $M, M'$ belong to the same $A$-orbit, then there is an integer $m$ such that $A^m(M)=M'$. Proposition \ref{Prop 5.1} then implies that $\tau^{4m}_{\CH}\left(\begin{smallmatrix}0\\ M \end{smallmatrix}\right)_0=\left(\begin{smallmatrix}0\\ A^mM \end{smallmatrix}\right)_0=\left(\begin{smallmatrix}0\\ M' \end{smallmatrix}\right)_0$, that is to say,  $\left(\begin{smallmatrix}0\\ M \end{smallmatrix}\right)_0$ and $\left(\begin{smallmatrix}0\\ M' \end{smallmatrix}\right)_0$ lie in the same $\tau_{\CH}$-orbit, thus also in the same component of $\Gamma_{\CH}$. A similar argument also shows the well-definition of $\beta$.

\vspace{.1 cm}

Further, we denote by $\mathcal{T}_{\infty}$ (resp. $\mathcal{T}_{\infty}'$) the subset of $\mathcal{T}$ (resp. $\mathcal{T}'$) consisting of all infinite components, and by $\mathcal{E}_{\infty}$ (resp. $\mathcal{E}_{\infty}'$) the inverse image of  $\mathcal{T}_{\infty}$ (resp. $\mathcal{T}_{\infty}'$)  under the map $\delta$ (resp. $\beta$). We recall that  the set of vertices of a stable tube  having exactly one immediate predecessor (or, equivalently, exactly one immediate successor) is called the {\it mouth} of the tube.

\begin{proposition}\label{proposition 3.7}
Suppose $\Lambda$ is an indecomposable self-injective Artin algebra. Under the above notation, the following statements hold.
\begin{itemize}
\item [$(1)$] The maps $\delta$ and $\beta$ are surjective.
\item [$(2)$]  The restricted  maps  $\delta\mid:\mathcal{E}_{\infty}\rt \mathcal{T}_{\infty}$ and  $\beta\mid:\mathcal{E}'_{\infty}\rt \mathcal{T}'_{\infty}$ are  bijections.
\end{itemize}
	\end{proposition}
\begin{proof}
The first assertion is clear from the definitions.	For the second one, assume that the component $\Gamma_{\CH}(M)$ is infinite for an indecomposable  non-projective $\Lambda$-module $M$. Hence by Proposition \ref{Theorem 6.9} $\Gamma_{\CH}(M)$ is a stable tube and the $\tau_{\CH}$-orbit of $\left(\begin{smallmatrix}0\\ M \end{smallmatrix}\right)_0$ generates all the vertices in the mouth of this tube.  Taking into account that the mouth of any stable tube is unique, it follows that $\Gamma_{\CH}(M)$ is uniquely determined by $[M]_{A}$. This gives the required result.
\end{proof}

\begin{remark}
We refer to a recent work \cite{HZ} of the first-named author for an observation similar to Proposition \ref{proposition 3.7} in the framework of the monomorphism category $\mathcal{S}$.
\end{remark}

\end{document}